\documentclass[a4paper,reqno]{amsart}

\usepackage[utf8]{inputenc}
\usepackage{amssymb}
\usepackage{enumitem}
\usepackage{color}
\usepackage{mathrsfs}
\usepackage{mathtools}
\setlist[enumerate]{label=\emph{(\roman*)}}

\usepackage{hyperref}

\newtheorem{theorem}{Theorem}[section]
\newtheorem{corollary}[theorem]{Corollary}
\newtheorem{lemma}[theorem]{Lemma}
\newtheorem{proposition}[theorem]{Proposition}

\theoremstyle{definition}
\newtheorem{definition}[theorem]{Definition}
\newtheorem{remark}[theorem]{Remark}

\numberwithin{equation}{section}

\newcommand\R{\mathbb{R}}
\newcommand\e{\epsilon}
\newcommand\om{\omega}
\newcommand\g{\gamma}

\begin{document}

\parindent=0pt

\title[Stability of the two-dimensional point vortices in Euler flows]
{Stability of the two-dimensional point vortices in Euler flows}

\author[D. Guo]{Dengjun Guo}
\address{School of Mathematical Sciences,
University of Science and Technology of China, Hefei 230026, Anhui, China}
\email{guodeng@mail.ustc.edu.cn}

\thanks{D.~Guo would like to thank L.~Zhao for his fruitful discussions and kind advices on this work. D.~Guo was partially supported by NSFC grant of China  (NO. 11771415).}
\begin{abstract}
We consider the two-dimensional incompressible Euler equation
\[\begin{cases} \partial_t \om + u\cdot \nabla \om=0 \\
\om(0,x)=\om_0(x).
\end{cases}\]
We are interested in the cases when the initial vorticity has the form $\om_0=\om_{0,\e}+\om_{0p,\e}$, where $\om_{0,\e}$ is concentrated near $M$ disjoint points $p_m^0$ and $\om_{0p,\e}$ is a small perturbation term.

  First, we prove that for such initial vorticities, the solution $\om(x,t)$ admits a decomposition $\om(x,t)=\om_{\e}(x,t)+\om_{p,\e}(x,t)$, where $\om_{\e}(x,t)$ remains concentrated near $M$ points $p_m(t)$ and $\om_{p,\e}(x,t)$ remains small for $t \in [0,T]$. Second, we give a quantitative description when the initial vorticity has the form $\om_0(x)=\sum_{m=1}^M \frac{\gamma_m}{\e^2}\eta(\frac{x-p_m^0}{\e})$, where we do not assume $\eta$ to have compact support. Finally, we prove that if $p_m(t)$ remains separated for all $t\in[0,+\infty)$, then $\om(x,t)$ remains concentrated near $M$ points at least for $t \le c_0 |\log A_{\e}|$, where $A_{\e}$ is small and converges to $0$ as $\e \to 0$.
\end{abstract}

\maketitle

\section{Introduction}
\subsection{Setting of the problem.}
We consider the two-dimensional incompressible Euler equation
\begin{equation}\begin{cases}\begin{aligned}\label{eq euler0}
\partial_t u+u \cdot \nabla u=&\nabla p \quad \text{in $[0,+\infty)\times \R^2$} \\
\nabla \cdot u=&0 \quad \,\,\,\,\,\text{in $[0,+\infty)\times \R^2$} \\
u(\cdot,0)=&u_0 \quad \,\,\text{in $\R^2$}.

\end{aligned}\end{cases}\end{equation}
Here $u=(u^1,u^2):\R^2 \to \R^2$ designates the velocity field and its vorticity $\om$ is defined by
$$
\om:=\partial_1 u^2-\partial_{2}u^1.
$$
We also have the  vorticity-stream formulation of the Euler equation
\begin{equation}\begin{cases}\begin{aligned}\label{eq euler}
\om_t+u \cdot \nabla \om&=0 \,\,\,\,\,\,\,\, \text{in}\,\, \R^2 \times [0,+\infty) \\
\om(\cdot,0)&=\om_0 \,\,\,\,\, \text{in}\,\, \R^2 ,
\end{aligned}\end{cases}\end{equation}
where the velocity $u$ can be recovered by vorticity $\om$ in view of the Biot-Savart law
$$u(x,t)=\frac{1}{2\pi} \int \frac{(x-y)^{\perp}}{|x-y|^2}\om(y,t)\,dy.$$
Yudovich proved in \cite{Yud} that equation \eqref{eq euler} is globally well-posed in some weak sense for initial vorticity  $\om_0\in L^1 \cap L^{\infty}$. That is, for any $\om_0 \in L^1 \cap L^{\infty}$, there exists a unique solution $\om(x,t)$ that solves equation \eqref{eq euler}. The solution $\om(x,t)$ satisfies
\begin{equation}\label{eq transport invariance}
\om(X(\alpha,t),t)=\om_0(\alpha)
\end{equation}
for any $\alpha \in \R^2$ and $t \in [0,+\infty)$. Here $X(\alpha,t)$ is the particle trajectory map which satisfies
\begin{equation*}\begin{cases}\begin{aligned}
\frac{dX(\alpha,t)}{dt}&= u(X(\alpha,t),t)\\
X(\alpha,0)&=\alpha.
\end{aligned}\end{cases}\end{equation*}
Moreover, for any time $t$, the map $X(\cdot,t):\R^2 \to \R^2 $ is a measure preserving diffeomorphism. We will consider the vortex solution of the Euler equation \eqref{eq euler} in this paper. Loosely speaking, we are interested in the initial vorticity of the form
$$\om_0(x)=\om_{0,\e}(x) \to \sum_{m=1}^M \g_m \delta_{p_m^0},$$
where $\g_1, ..., \g_M$ are $M$ nonzero real numbers and $p^0_1, ... , p^0_M \in \R^2$ are $M$ distinct points. A computation in \cite{Hel} suggests that when
\begin{equation*}\om_0(x)=\sum_{m=1}^M \g_m \delta_{p_m^0},\end{equation*}
the solution $\om(x,t)$ of \eqref{eq euler0} should be
$$
\om(x,t)=\sum_{m=1}^M \g_m \delta_{p_m(t)},
$$
where $p(t)=\big(p_1(t), ..., p_M(t)\big):[0,T^*) \to (\R^2)^M$ solves the Helmholtz equation
\begin{equation}\begin{cases}\begin{aligned}\label{eq helmholtz}
\frac{dp_m}{dt}&=\sum_{l \neq m} \frac{\g_l}{2\pi} \frac{(p_m-p_l)^{\perp}}{|p_m-p_l|^2} \\
p_m(0)&=p_m^0.
\end{aligned}\end{cases}\end{equation}
Let us introduce some previous researches on Helmholtz equation first. Unlike the incompressible Euler equation, the solution $p(t)$ to equation \eqref{eq helmholtz} may blow-up in finite time for some initial data,  see e.g. \cite{Are}, \cite{MaP2}. However, it can be shown that the initial data for ODE \eqref{eq helmholtz} which produces a blow-up is exceptional, see \cite{MaP2} and \cite{Dur} for references.
Many mathematicians have investigated the connection between the point vortex dynamic \eqref{eq helmholtz} and the incompressible Euler equation with concentrated initial vorticity; see e.g. \cite{MaP1}, \cite{Wei}, \cite{Mar}, \cite{MaP3}, \cite{MaP11}, \cite{MaP10}, \cite{MaP}, \cite{Ser1}, \cite{Ser2}, \cite{Tuk}. The first rigorous justification is due to Marchioro and Pulvirenti \cite{MaP} and then improved by Marchioro \cite{MaP3}, \cite{MaP11} and Serfati \cite{Ser2}. Let $p(t)$ be the solution of \eqref{eq helmholtz} and $T^*$ be its maximal time of existence, Serfati proved in \cite{Ser2} that when the initial vorticity $\om_{0,\e}(x)$ satisfies
\begin{enumerate}
\item $\om_{0,\e}(x)=\sum_{m=1}^M \om_{0m,\e}(x),$
\item $\int \om_{0m,\e} \,dx=\g_m,$
\item $\g_m \om_{0m,\e}(x) \ge 0$ for all $x \in \R^2$ and $m=1, ..., M$,
\item $\text{supp}(\om_{0m,\e}) \subset B_{\e}(p_m^0),$
\item $\|\om_{0m,\e}\|_{\infty} \le \e^{-k}$ for some $k>0$.
\end{enumerate}
then for any $\alpha<\frac12$ and $T<T^*$, there exists $\e_0=\e_0(k,\alpha,T)$ such that when $\e \le \e_0$, the solution $\om_{\e}(x,t)$ to the Euler equation \eqref{eq euler} admits the following decomposition for all $t \in [0,T]$:
$$\om_{\e}(x,t)=\sum_{m=1}^M \om_{m,\e}(x,t).$$
The $m^{th}$ vorticity $\om_{m,\e}(\cdot,t)$ is supported in $B_{\e^{\alpha}}\big(p_m(t)\big)$ and satisfies $$\int \om_{m,\e}(x,t) \,dx=\g_m.$$
By a similar argument, the $L^{\infty}$ norm in (v) can be replaced by $L^p$ norm for any $p>2$, see \cite{Jer} for more details.
 For the initial vorticity which do not have compact support, Marchioro and Pulvirenti proved in \cite{MaP} that when $\om_{0m,\e}\ge 0$ for all $m$ and $\om_{0,\e} \to \sum \gamma_m \delta_{p_m^0}$ in distribution, then $\om_{\e}(\cdot,t) \to \sum \gamma_m\delta_{p_m(t)}$ in distribution. Moreover, they showed that when $\om_{0m,\e}(x)$ is uniformly bounded outside $B_{\e}(p_m^0)$, $\|\om_{0m,\e}\|_{\infty} \le \e^{-k}$ for some $k< \frac{8}{3}$ and $\om_{0m,\e}\to \gamma_m \delta_{p_m^0}$, then we also have $\om_{\e}(\cdot,t) \to \sum \gamma_m\delta_{p_m(t)}$. In \cite{Wei}, Dávila, del Pino, Musso and Wei considered the initial vorticity of the form
$$\om_{0,\e}(x)\approx \sum_{m=1}^M\frac{\gamma_m}{\e^2}\eta(\frac{x}{\e}),$$
where $\eta(y)=\frac{1}{(1+|y|^2)^2}$. They used a gluing method to give a detailed description of the solution near the point $p_m(t)$. More precisely, they showed that the solution $\om_{\e}(x,t)$ to the Euler equation admits a decomposition: $$\om_{\e}(x,t)=\sum_{m=1}^M \frac{\gamma_m}{\e^2} \eta (\frac{x-p_m(t)}{\e})+\phi_{\e}(x,t),$$ where $\phi_{\e}(x,t)$ is a small perturbation term in some suitable senses.

    \quad Recall that the solution $p(t)$ to the equation \eqref{eq helmholtz} is global for general initial data $p^0$, so one might ask what will happen for the Euler equation when the initial vorticity is concentrated near the point $p^0$. For such initial vorticities, Buttà and Marchioro proved in \cite{MaP1} that the solution $\om_{\e}(x,t)$ remains concentrated near $p(t)$ during a time scale $t \approx |\log \e|$. They also show that when $M=1, p_1^0=0$ (a single vortex concentrated near origin) and if the space domain $\R^2$ is replaced by $B_R(0)$, then the solution remains concentrated near origin during a time scale $t \approx \e^{-a}$ for some $a>0$. In \cite{Mar}, Donati and Iftimie generalized the result to simply connected bounded domains.

      \quad For three-dimensional Euler equation, there is a corresponding phenomenon which called vortex rings. More precisely, if the initial vorticity is supported near a ring in $\R^3$, then the solution remains concentrated near a traveling ring with speed $|\log \e|$, see  \cite{MaP8}, \cite{MaJ}, \cite{MaP6}, \cite{MaP5}, \cite{MaP9}, \cite{Jer} and \cite{MaP7} for references. See also \cite{Jer2}, \cite{Wei1} for a similar phenomenon in Gross-Pitaevskii equations.

\subsection{Main results.}We first introduce some basic notions. For a point $p_m \in \R^2$, we denote by $B_{\e}(p_m)$ the open ball centered at $p_m$ with radius $\e$. And for $f$ a measurable function on $\R^2$, we denote by $\|f\|_p$ the $L^p$ norm of $f$.  Throughout the paper, $C$ denotes a constant which might change from line to line. We denote by $c_0$ a small constant and by $c_0(\gamma)$ a small constant which might depending on $\gamma$.\\
  \\We consider the two-dimensional Euler equation of the form
\begin{equation}\begin{cases}\label{eq perturb}
\partial_t \omega_{\e}+u_{\e}\cdot \nabla \omega_{\e} =0 \\
\omega_{\e}(x,0)=\om_{0\e}(x)=\sum_{m=1}^M \omega_{0m,\e}(x)+\omega_{0p ,\e}(x),
\end{cases}\end{equation}
where
\[
u_{\e}(x,t)=\frac{1}{2\pi}\int \frac{(x-y)^{\perp}}{|x-y|^2} \om_{\e}(y,t)\,dy.
\]
Fix $q>2$ and set
$$A_{\e}:= \max \left\{   \|\om_{0p,\e}\|_q^{\frac{q}{2q-2}}\| \om_{0p,\e}\|_1^{\frac{q-2}{2q-2}}, \e \right\}. $$

Let $p(t)=\left(p_1(t), ..., p_M(t)\right)$ be the solution of equation \eqref{eq helmholtz}. Assume $p_m(t) \neq p_l(t)$ for all $t\in [0,T^*)$ when $m\neq l$, then we have
\begin{theorem}\label{thm main}
Assume $\om_{0,\e} \in L^{1} \cap L^{\infty}$ and $\om_{0m,\e}$ does not change sign with
\[\text{supp}(\om_{0m,\e}) \subset B_{\e}(p_m(0)).\]
 Let the $m^{th}$ total vorticity $\gamma_{m,\e}$ be $$\gamma_{m,\e}= \int \om_{0m, \e} \,dx,$$
which satisfies
 \[
 |\gamma_m-\gamma_{m,\e}| \le CA_{\e}^{\frac12}.
 \]
We assume also that $A_{\e} \to 0$ as $\e \to 0$ and there exist some $\gamma >0$ , $p_1>2$ such that $$\| \om_{0m,\e}\|_{p_1} \le A_{\e}^{-\gamma}$$
for $m=1, ... , M$.
Then for any $T < T^*$ and any $a < \frac12$, there exists $\e_0 >0 $ such that for all $0 < \e \le \e_0$, the solution $\om_{\e}$ of \eqref{eq perturb} admits a decomposition $\om_{\e}=\sum_{m=1}^M \om_{m,\e}+\om_{p,\e}$, where
$$\text{supp}(\om_{m,\e}(\cdot,t))\subset B_{(A_{\e})^{a}}(p_m(t)),$$ $$ \int \om_{m,\e}(x,t) \,dx=\gamma_{m,\e}$$
and
$$\|\om_{p,\e}\|_r=\|\om_{0p,\e}\|_r$$ for any $r \ge 1$.
\end{theorem}

\begin{remark}
For the initial vorticity $\om_{0\e}$ as in \eqref{eq perturb}, we see directly that $\om_{0\e,m}$ corresponds to an $O(\e)$ regularization of the Dirac. We also note that $\om_{0\e,p}$ generates a velocity field $u_{0\e,p}$ which corresponding to an $O(\|u_{0\e,p}\|_{\infty})$ perturbation of the Euler flow.  So $A_{\e}$ is a natural quantity that is used to measure the competition between the regularization and the perturbation.
\end{remark}
Next we consider the initial vorticity $\om_{0,\e}$ of the form
$$\om_{0,\e}(x)=\sum_{m=1}^M \frac{\gamma_m}{\e^2} \eta(\frac{x-p_m^0}{\e}).$$
To ensure that $\om_{0,\e} \in L^1 \bigcap L^{\infty}$, we make an assumption on $\eta$ that
$
\eta(x) \lesssim \frac{1}{1+|x|^{2+\sigma}}
$ for some $\sigma>0$. We prove that under such assumptions, the solution $\om_{\e}(x,t)$ of the Euler equation remains concentrated near $M$ points $p_m(t)$ for any $0 \le t \le T$. Here we do not require $\eta$ to have compact support.
\begin{theorem}\label{thm concentration of mass}(Concentration of the vorticity)
Let $\eta \in L^1 \bigcap L^{\infty}$ be a nonnegative function with
\[
\eta(x) \le \frac{C}{1+|x|^{2+\sigma}}
\]
for some $\sigma>0$ and
\[
\int \eta(x) \,dx=1.
\]
Suppose the initial vorticity $\om_{0,\e}$ has the form
$$\om_{0,\e}(x)=\sum_{m=1}^M \frac{\gamma_m}{\e^2} \eta(\frac{x-p_m^0}{\e}).$$
Let $\om_{\e}(x,t)$ be the solution of the Euler equation with initial vorticity $\om_{0,\e}$, then we obtain
\begin{enumerate}
\item \emph{Concentration properties.} For $\sigma>0$ and any $a <\frac{\sigma}{2(\sigma+2)}$. Let
\[
B(t):= \bigcup_{m=1}^M B_{\e^a}(p_m(t)).
\]
Then we have
\[
\int_{B(t)^c} |\om_{\e}(x,t)|\,dx\to 0 \quad \text{as} \quad \e \to 0.
\]
More precisely, for $\sigma>0$ and $a <\frac{\sigma}{2(\sigma+2)}$, there exists $\e_0=\e_0(a,\sigma,T)$ such that
\begin{equation}\label{eq concentration of mass1}
\int_{B(t)^c} |\om_{\e}(x,t)|\,dx \le C \e^{\frac{\sigma^2}{\sigma+2}}
\end{equation}
for all $0<\e \le \e_0$ and $0 \le t \le T$.

\item \emph{Smallness of the vorticity outside the core.} For any $\gamma_1<\sigma$ there exist constants $a,\e_0>0$ such that
\begin{equation*}
\int_{B(t)^c} |\om_{\e}(x,t)|\,dx \le \e^{\gamma_1}
\end{equation*}
for all $0<\e \le \e_0$ and $0 \le t \le T$.
More precisely, for $\sigma>0$ and any $\frac{\sigma}{\sigma+1}<\gamma_1<\sigma$, define
\[
a_0(\gamma_1,\sigma):= \frac{1}{2}\min \left\{1-\frac{\gamma_1}{\sigma},\gamma_1+\frac{\gamma_1}{\sigma}-1 \right\}.
\]
Then  for any $a<a_0$, there exists $\e_0=\e_0(a,\sigma,\gamma_1,T)$ such that
\begin{equation}\label{eq concentration of mass2}
\int_{B(t)^c} |\om_{\e}(x,t)|\,dx \le C\e^{\gamma_1}
\end{equation}
for all $0<\e \le \e_0$ and $0 \le t \le T$.
\end{enumerate}

\end{theorem}
\begin{remark}\label{rmk concentration of mass}
Indeed, we can show that $\om_{\e}(\cdot,t)$ admits a decomposition
\[
\om_{\e}(\cdot,t)=\sum_{m=1}^M \om_{m,\e}(\cdot,t)+\om_{p,\e}(\cdot,t)
\]
where $\om_{m,\e}(\cdot,t)$ is supported in $B_{\e^a}(p_m(t))$ and $\om_{p,\e}$ satisfies \eqref{eq concentration of mass1} or \eqref{eq concentration of mass2} respectively.
\end{remark}
\begin{remark}
We can also consider the initial vorticity of the form
$$\om_{0,\e}(x)=\sum_{m=1}^M \frac{\gamma_m}{\e^2} \eta_{m,\e}(\frac{x-p_m^0}{\e}),$$
where $\eta_{m,\e}$ is a nonnegative function satisfies $\int \eta_{m,\e}(x)\,dx=1$ and $\eta_{m,\e}(x) \lesssim \frac{1}{1+|x|^{2+\sigma}}$. The same result as Theorem \ref{thm concentration of mass} holds and the proof is identically the same.
\end{remark}
\begin{remark}{(Stability of the solution.)}
With the same method, we can treat the initial vorticity of the form
\[
\om_{0,\e}(x)=\sum_{m=1}^M \frac{\gamma_m}{\e^2} \eta_{m,\e}(\frac{x-p_m^0}{\e})+\om_{0p,\e}(x),
\]
where $\om_{p,\e}$ satisfies
\[
\|\om_{0p,\e}\|_q^{\frac{q}{2q-2}}\| \om_{0p,\e}\|_1^{\frac{q-2}{2q-2}} \le \e^{a}
\]
for some $q>2$ and $a>0$. For such initial vorticities, a similar result holds as in Theorem \ref{thm concentration of mass}.

\end{remark}

As a corollary, we get
\begin{corollary}\label{cor concentration of mass}(Concentration of the vorticity for Schwartz functions) Let $\eta$ be a nonnegative Schwartz function with
$$
\int \eta(x) \,dx=1.
$$
Assume the initial vorticity $\om_{0,\e}$ has the form
$$\om_{0,\e}(x)=\sum_{m=1}^M \frac{\gamma_m}{\e^2} \eta(\frac{x-p_m^0}{\e}).$$
Then for any $a < \frac12$ and any $\gamma>0$, there exists $\e_0=\e_0(a,\gamma,T)$ such that when $\e \le \e_0$, we have
\[
\int_{B(t)^c} |\om_{\e}(x,t)|\,dx \le \e^{\gamma}
\]
for all $t \in [0,T]$.
\end{corollary}

  If we assume also that $|p_m(t)-p_l(t)|\ge \delta>0$ for all $t \in [0,\infty)$ when $m\neq l$. Then we obtain the following Theorem.
\begin{theorem}\label{thm long time}(Long time dynamic for Euler flows.) Assume $\om_{0,\e}$ satisfies the same assumptions as in Theorem \ref{thm main}. For $a<\frac12$, let
$$\tau_{\e} := \sup \Big\{ t\in [0,\infty) : \text{supp}\big(\om_{m,\e}(\cdot,t)\big)\subset B_{A_{\e}^a}\big(p_m(t)\big) \Big\}.$$

Then there exist $c_0=c_0(a,\delta)>0$ and $\e_0=\e_0(c_0)$ such that
$$
\tau_{\e} \ge c_0 |\log A_{\e}|
$$
for all $0 < \e \le \e_0$. Similar results of Theorem \ref{thm concentration of mass} and Corollary \ref{cor concentration of mass} hold as well.
\end{theorem}

\subsection{Strategy of the proof.}
We will follow the strategy in \cite{MaP}, in which Marchioro and Pulvirenti considered the Euler equation with a Lipschitz continuous background flow
\begin{equation*}
\begin{cases}\begin{aligned}
\partial_t \om_{\e}+(u_{\e}+F_{\e}) \cdot \nabla \om_{\e}=0 \\
\om_{\e}(\cdot,0)=\om_{0,\e}.
\end{aligned}\end{cases}
\end{equation*}
Here $\nabla \cdot F_{\e}=0,$
 $F_{\e}(\cdot,t)$ is Lipschitz continuous and the Lipschitz norm is independent of time $t$. In Section 2, we will show that the presence of the perturbation term $\om_{0p,\e} $ may lead to an equation of the form
\begin{equation}\label{eq section1 single vortex MaP1}
\begin{cases}\begin{aligned}
\partial_t \om_{\e}+(u_{\e}+F_{1,\e}+F_{2,\e}) \cdot \nabla \om_{\e}=0 \\
\om_{\e}(\cdot,0)=\om_{0,\e},
\end{aligned}\end{cases}
\end{equation}
where $F_{1,\e}$ satisfies the same property as $F_{\e}$ and $F_{2,\e}$ is a divergence free perturbation term related to $\om_{0p,\e}$. Next we show that when the initial vorticity $\om_{0,\e}(x)$ is supported in a ball of radius $\e$, then the solution $\om_{\e}(x,t)$ of \eqref{eq section1 single vortex MaP1} remains supported in a small ball near its center of gravity. By a standard bootstrap argument and the prior estimates we made in Section 2, Theorem \ref{thm main} and Theorem \ref{thm long time} follow directly as in \cite{MaP1},  \cite{Jer} and \cite{MaP},   and we will give full details for completeness. In Section 3 we consider the initial vorticity of the self-similar form
\begin{equation*}\om_{0,\e}(x)=\sum_{m=1}^M \frac{\gamma_m}{\e^2} \eta(\frac{x-p_m^0}{\e}).\end{equation*}
We will show that this kind of initial vorticity admits a decomposition of the form
\[
\om_{0,\e}(x)=\sum_{m=1}^M \omega_{0m,\e}(x)+\omega_{0p ,\e}(x),
\]
which satisfies all the assumptions in Theorem \ref{thm main}.
So the concentration property of the vorticity follows directly from Theorem \ref{thm main}.

\subsection{Outline of the paper.}
In Section 2 we consider a single vortex with a background flow and obtain some prior estimates. In Section 3.1 we discuss the stability of the point vortices in Euler flows and prove Theorem \ref{thm main}. In Section 3.2, we consider the Euler equation with self-similar initial data and prove Theorem \ref{thm concentration of mass}. In Section 4 we discuss the long time behavior for the point vortex dynamics and prove Theorem \ref{thm long time}. In Section 5 we discuss some possible generalizations for the point vortices in a smooth bounded domain.

\section{Control of each point vortex. }
\subsection{Decomposition of the vorticity.}
Instead of considering equation \eqref{eq perturb}, we first consider more generally the following equation:
\begin{equation}\begin{cases}\begin{aligned}\label{eq euler2}
\om_t+u \cdot \nabla \om&=0 \,\,\,\,\,\,\,\, \text{in}\,\, \R^2 \times [0,+\infty) \\
\om(\cdot,0)&=\om_0=\sum_{m=1}^M \om_{0m} \,\,\,\,\, \text{in}\,\, \R^2,
\end{aligned}\end{cases}\end{equation}
where
\[
u(x,t)=\frac{1}{2\pi}\int \frac{(x-y)^{\perp}}{|x-y|^2}\om(y,t)\,dy
\]
as usual.

\begin{definition}
Let $F:\R^2 \times [0,T)$ satisfy $\nabla \cdot F=0$. Then we say that $\om$ satisfies
\[\begin{cases}
\partial_t \om+F \cdot \nabla \om=0 \\
\om(\cdot,0)=u_0
\end{cases}\]
in weak sense if for any $\phi \in C^1 \left([0,T),C_c^2(\R^2)\right)$
\[
\frac{d}{dt} \int \om(x,t)\phi(x,t) \,dx=\int \om(x,t)\left(\partial_t \phi +F \cdot \nabla \phi \right) \,dx.
\]
\end{definition}
\begin{lemma}\label{lemma euler decomposition}(Decomposition of the vorticity)
Let $u$ be the unique Yudovich solution of equation \eqref{eq euler2}. Assume $\om_{0m}\in L^1 \cap L^{\infty}$ for $m=1, ..., M$. Define $$\om_{m}(X(\alpha,t),t):=\om_{0m}(\alpha),$$ where $X(\alpha,t)$ satisfies the differential equation
\begin{equation*}\begin{cases}
\frac{dX(\alpha,t)}{dt}=u(X(\alpha,t),t) \\
X(\alpha,0)=\alpha.
\end{cases}\end{equation*}

 Then $\om_m(\cdot,t) $ is a rearrangement of $\om_{0m}(\cdot)$ and satisfies the following equation in weak sense:
\begin{equation}\begin{cases}\label{eq euler decomposition}
\partial_t \om_m+(\sum_{m=1}^M u_m) \cdot \nabla \om_m=0 \\
\om_m(\cdot,0)=\om_{0m},
\end{cases}\end{equation}
where
$$u_m(x,t)=\frac{1}{2\pi} \int \frac{(x-y)^{\perp}}{|x-y|^2}\om_m(y,t)\,dy.$$
\end{lemma}

\begin{proof}
Since $X(\cdot,t):\R^2 \to \R^2$ is a measure preserving diffeomorphism, so we see directly that $\om_m(\cdot,t) $ is a rearrangement of $\om_{0m}(\cdot)$. Thus, it remains to prove that $\om_{m}(x,t)$ satisfies equation \eqref{eq euler decomposition} in weak sense. Fix $\phi \in C^1([0,T),C_c^2(\R^2))$, since the map $x \to X(\alpha,t)$ is measure preserving, we make the change of variable $x\to X(\alpha,t)$ and compute that
\[\begin{aligned}\frac{d}{dt}\int \om_m(x,t)\phi(x,t)\,dx &=\frac{d}{dt} \int \om_{0m}(\alpha)\phi(X(\alpha,t),t) \, d\alpha \\
&=\int
\om_{0m}(\alpha) \left(\partial_t \phi+u \cdot \nabla \phi\right)(X(\alpha,t),t) \,d\alpha \\
&=\int \om_m(x,t)\left(\partial_t \phi+u \cdot \nabla \phi\right)(x,t) \,dx.
\end{aligned}\]
Then by the vorticity transport formula \eqref{eq transport invariance} and the definition of $\om_m(x)$, we see that the solution $\om(x,t)$ of equation \eqref{eq euler2} is exactly $\sum_{m=1}^M \om_m(x,t)$. Therefore,
\[\begin{aligned}
u(x,t)&=\frac{1}{2\pi} \int \frac{(x-y)^{\perp}}{|x-y|^2}\om(y,t)\,dy \\
&=\frac{1}{2\pi} \int \frac{(x-y)^{\perp}}{|x-y|^2}(\sum_{m=1}^M\om_m(y,t))\,dy \\
&=\sum_{m=1}^M u_m(x,t) .
\end{aligned}\]
Thus we obtain
\begin{equation}\label{eq integration by part}\begin{aligned}
\frac{d}{dt}\int \om_m(x,t)\phi(x,t)\,dx &=\int \om_m(x,t)\left(\partial_t \phi+(\sum_{m=1}^M u_m) \cdot \nabla \phi\right)(x,t) \,dx
\end{aligned}\end{equation}

which shows that $\om_m(x,t)$ satisfies \eqref{eq euler decomposition} in the weak sense.
\end{proof}

  In order to control the velocity field $u(x,t)$ by its vorticity $\om(x,t)$, we prove the following estimates.
\begin{lemma}\label{lemma control of the velocity}
Assume $f \in L^1 \cap L^q$ for some $q>2$. Then we get
$$\Big|\frac{1}{2\pi} \int \frac{(x-y)^{\perp}}{|x-y|^2}f(y)\,dy\Big| \lesssim_q \|f\|_1^{\frac{q-2}{2q-2}} \|f\|_q^{\frac{q}{2q-2}} $$
\end{lemma}
\begin{proof}
We assume that $f \neq 0$, then for any $R>0$, H\"{o}lder's inequality gives
\[\begin{aligned}
\big|\frac{1}{2\pi} \int \frac{(x-y)^{\perp}}{|x-y|^2}f(y)\,dy\big| \lesssim&  \int \frac{1}{|x-y|}|f(y)|\,dy \\
\lesssim&  \int_{B_R(x)} \frac{1}{|x-y|}|f(y)|\,dy+\int_{B_R(x)^c} \frac{1}{|x-y|}|f(y)|\,dy \\
\lesssim& \|f\|_q \int_{B_R(x)} |x-y|^{-q'} \,dy^{\frac{1}{q'}}+\frac{\|f\|_1}{R} \\
\lesssim& (\frac{1}{2-q'})^{\frac{1}{q'}}R^{\frac{2-q'}{q'}}\|f\|_q + \frac{\|f\|_1}{R} \\
\lesssim&_q R^{\frac{2-q'}{q'}}\|f\|_q + \frac{\|f\|_1}{R},
\end{aligned}\]
where $q'$ satisfies $\frac{1}{q'}=1-\frac{1}{q}$. Choosing $R=(\frac{\|f\|_1}{\|f\|_q})^{\frac{q'}{2}}$, then we get the desired inequality.
\end{proof}

\begin{remark}\label{remark compact integration by part}
Assume $\om_{0m}$ has compact support for some $m$, then by Lemma \ref{lemma control of the velocity} and the vorticity transport formula \eqref{eq transport invariance}, we see that the support of $\om_m(\cdot,t)$ remains bounded for any time $t <+\infty$. So \eqref{eq integration by part} holds for any function $\phi \in C^1([0,+\infty),C^2(\R^2))$ .
\end{remark}

After the above preliminary work, we now consider the Euler equation \eqref{eq perturb}
\[\begin{cases}
\partial_t \omega_{\e}+u_{\e}\cdot \nabla \omega_{\e} =0 \\
\omega_{\e}(x,0)=\om_{0\e}(x)=\sum_{m=1}^M \omega_{0m,\e}(x)+\omega_{0p ,\e}(x).
\end{cases}\]
 Apply Lemma \ref{lemma euler decomposition} to $\om_0(x)=\sum_{m=1}^M \omega_{0m,\e}(x)+\omega_{0p ,\e}(x)$, then $\om_{\e}(x,t)$ admits a decomposition
 \[
 \om_{\e}(x,t)=\sum_{m=1}^M\om_{m,\e}(x,t)+\om_{p,\e}(x,t),
 \]
 where $\om_{m,\e}(x,t)$ satisfies
\begin{equation}\begin{cases}\begin{aligned}\label{eq single vortex1}
\partial_t \om_{m,\e}+(\sum_{m=1}^M u_{m,\e}+u_{p,\e}) \cdot \nabla \om_{m,\e}&=0 \\
\om_{m,\e}(\cdot,0)&=\om_{0m,\e}
\end{aligned}\end{cases}\end{equation}
for $m=1, ..., M$ and $\om_{p,\e}(x,t)$ satisfies
\begin{equation*}\begin{cases}\begin{aligned}
\partial_t \om_{p,\e}+(\sum_{m=1}^M u_{m,\e}+u_{p,\e}) \cdot \nabla \om_{p,\e}&=0 \\
\om_{p,\e}(\cdot,0)&=\om_{0p,\e}.
\end{aligned}\end{cases}\end{equation*}

Again by Lemma \ref{lemma euler decomposition}, $\om_{p,\e}(\cdot,t)$ is a rearrangement of $\om_{0p,\e}(\cdot)$. Therefore,
$$\|\om_{p,\e}(\cdot,t)\|_r=\|\om_{0p,\e}\|_r$$
for any $1\le r \le +\infty$. Finally, we apply Lemma \ref{lemma control of the velocity} to $f=\om_{p,\e}$, and obtain
\begin{equation*}
\|u_{p,\e}\|_{\infty} \lesssim_q \|\om_{0p,\e}\|_1^{\frac{q-2}{2q-2}} \|\om_{0p,\e}\|_1^{\frac{q}{2q-2}} \lesssim_q A_{\e},
\end{equation*}
where
$$A_{\e}= \max \left\{   \|\om_{0p,\e}\|_q^{\frac{q}{2q-2}}\| \om_{0p,\e}\|_1^{\frac{q-2}{2q-2}}, \e \right\}$$ is defined in section 1. With the above estimates on $\om_{p,\e}(x,t)$, next we aim to show that equation \eqref{eq single vortex1} can be written in a form like \eqref{eq section1 single vortex MaP1}. We first prove a technical lemma:
\begin{lemma}\label{lemma lipschitz property}
Let $\delta>0$ be a positive number. Assume $\om(x) \in L^1 \cap L^{\infty}$ and $\text{supp}(\om) \subset B_{\delta/4}(p)$ for some point $p \in \R^2$, then for any $x\in \R^2$ such that $|x-p| \ge \delta$, we have
$$|\nabla u(x) | \le \frac{C}{\delta^2} \| \om \|_1,$$
where
$$u(x)=\frac{1}{2\pi} \int \frac{(x-y)^{\perp}}{|x-y|^2} \om(y)\,dy.$$
\end{lemma}
\begin{proof} Due to the support of $\om$, we have $|x-y|\ge \frac{\delta}{2}$ when $|x-p| \ge \delta$ and $y \in \text{supp}(\om)$. Thus,
\[\begin{aligned}
|\nabla u(x)| \lesssim& \int \frac{|\om(y)|}{|x-y|^2}\,dy=\int_{B_{\frac{\delta}{4}}(p)} \frac{|\om(y)|}{|x-y|^2}\,dy \lesssim \frac{1}{\delta^2} \|\om\|_1.
\end{aligned}\]
\end{proof}

Assume that $\om_m(\cdot,t)$ is supported in a small ball centered at the point $p_m(t)$, then equation \eqref{eq single vortex1} can be rewritten as
\begin{equation*}\begin{cases}\begin{aligned}
\partial_t \om_{m,\e}+(u_{m,\e}+F_{1,m,\e}+F_{2,m,\e}) \cdot \nabla \om_{m,\e}=0 \\
\om_{m,\e}(\cdot,0)=\om_{0m,\e}
\end{aligned}\end{cases}\end{equation*}

where $F_{2,m,\e}(\cdot,t)=u_{p,\e}(\cdot,t)$ is a small perturbation term and $$F_{1,m,\e}(\cdot,t)=\sum_{l\neq m}u_{l,\e}(\cdot,t)$$ is Lipschitz continuous on the support of $\om_{m,\e}(\cdot,t)$ by Lemma \ref{lemma lipschitz property}. The Lipschitz constant can be controlled independent of the time $t$ since $\|\om_{l,\e}(\cdot,t)\|_r$ is a conserved quantity for any $l=1, ..., M$ and $r\ge 1$.

\subsection{A single vortex with a background flow.}
Now we consider more generally the following problem:
Let $\om_{\e}(\cdot,t) \in L^1 \cap L^{\infty}$ be a rearrangement of $\om_{0,\e}$ and satisfies
\begin{equation}\begin{cases}\begin{aligned}\label{eq single vortex}
\partial_t \om_{\e}+(u_{\e}+F_{1,\e}+F_{2,\e}) \cdot \nabla \om_{\e}=0 \\
\om_{\e}(\cdot,0)=\om_{0,\e}
\end{aligned}\end{cases}\end{equation}
in weak sense. That is, for any smooth function $\phi(x,t)$,
\begin{equation}\label{eq integration by parts1}
\frac{d}{dt}\int \om_{\e}(x,t)\phi(x,t)\,dx=\int \om_{\e}(x,t)\left( \partial_t\phi(x,t)+(u_{\e}+F_{1,\e}+F_{2,\e})\cdot \nabla \phi(x,t) \right) \,dx.
\end{equation}

We assume that
\begin{enumerate}
\item $|F_{1,\e}(x,t)-F_{1,\e}(y,t)| \le L|x-y|$ for $L$ a positive number independent of $\e$ and $t$.
\item
$\|F_{2,\e}\|_{\infty}\to 0 \quad \text{as $\e \to 0$}.$
\item The vorticity transport formula holds:
$$\om_{\e}(X(\alpha,t),t)=\om_{0,\e}(\alpha),$$
where $X(\alpha,t)$ satisfies
\[\begin{cases}
\frac{dX(\alpha,t)}{dt}=(u_{\e}+F_{1,\e}+F_{2,\e})(X(\alpha,t),t) \\
X(\alpha,0)=\alpha.
\end{cases}\]
\item For any time $t$, the map $\alpha \to X(\alpha,t)$ is measure preserving.

\end{enumerate}
Then we have

\begin{theorem}\label{thm section2 main}
Assume $\om_{0,\e}$ has compact support and belongs to $L^1\cap L^{\infty}$. Let $\om_{\e}$, $F_{1,\e}$ and $F_{2,\e}$ satisfy the condition above. We assume also that $\om_{0,\e}$ does not change sign. Let $\Omega_{\e} :=\int_{\R^2}\om_{0,\e}\,dx$ and define
\begin{equation*}
p_{\e}(t)=\frac{1}{\Omega_{\e}}\int x \om_{\e}(x,t)\,dx,
\end{equation*}
\begin{equation*}
I_{\e}(t)=\frac{1}{2\Omega_{\e}}\int|x-p_{\e}(t)|^2 \om_{\e}(x,t)\,dx.
\end{equation*}
Let $P(t)$ solve
$$\frac{dP}{dt}=F_{1,\e}(P(t),t).$$
Then we have
\begin{equation}\label{eq section2 control of second momentum}
I_{\e}(t) \le 2e^{2Lt}\left[ I_{\e}(0)+\frac{\|F_{2,\e}\|_{\infty}^2}{2}(\int_0^t e^{-Ls}\,ds)^2\right]
\end{equation}
and
\begin{equation}\begin{aligned}\label{eq section2 control of center of gravity}
|p_{\e}(t)-P(t)|\le e^{Lt} \Bigg[ |p_{\e}(0)-P(0)|&+2L(\sqrt{I_{\e}(0)}+\frac{\|F_{2,\e}\|_{\infty}}{\sqrt{2}L})&\int_0^t \int_0^r e^{-Ls }\,ds\,dr \\&+\|F_{2,\e}\|_{\infty}\int_0^t e^{-Ls}\,ds \Bigg].
\end{aligned}\end{equation}
\end{theorem}
\begin{proof}
Let $K(z)=\frac{1}{2\pi} \frac{z^{\perp}}{|z|^2}$ be the Biot-Savart kernel, then the antisymmetry of $K$ implies
\begin{equation}\label{eq section2 antisymmetry}
\int_{\R^2} u_{\e} \om_{\e} \,dx=\int_{\R^2 \times \R^2} K(x-y)\om_{\e}(x,t)\om_{\e}(y,t) \,dxdy=0.
\end{equation}
Similarly, since $z \cdot K(z)=0$, we obtain
\begin{equation}\begin{aligned}\label{eq section2 antisymmetry2}
\int x \cdot u_{\e}(x,t)\om_{\e}(x,t)\,dx &= \int_{\R^2 \times \R^2} x \cdot K(x-y)\om_{\e}(x,t)\om_{\e}(y,t) \,dxdy \\
&= \frac{1}{2}\int_{\R^2 \times \R^2} (x-y)\cdot K(x-y)\om_{\e}(x,t)\om_{\e}(y,t) \,dxdy=0.
\end{aligned}\end{equation}
By definition of $p_{\e}(t)$, we note that for any vector $\nu(t)$,
\begin{equation}\label{eq section2 antisymmetry3}
\frac{1}{\Omega_{\e}}\int \nu(t) \cdot (x-p_{\e}(t))\om_{\e}(x,t) \,dx=\nu(t)\cdot \left(p_{\e}(t)-p_{\e}(t)\right)=0.
\end{equation}
  Then due to \eqref{eq integration by parts1}, \eqref{eq section2 antisymmetry3}, \eqref{eq section2 antisymmetry}, \eqref{eq section2 antisymmetry2}, we get
\[\begin{aligned}
\frac{dI_{\e}(t)}{dt}&=\frac{1}{\Omega_{\e}}\int(x-p_{\e}(t))\cdot\left(-\frac{dp_{\e}(t)}{dt}+u_{\e}(x,t)+F_{1,\e}(x,t)+F_{2,\e}(x,t)\right)\om_{\e}(x,t)\,dx \\
&=\frac{1}{\Omega_{\e}}\int(x-p_{\e}(t))\cdot\left(u_{\e}(x,t)+F_{1,\e}(x,t)+F_{2,\e}(x,t)\right)\om_{\e}(x,t)\,dx \\
&=\frac{1}{\Omega_{\e}}\int(x-p_{\e}(t))\cdot\left(F_{1,\e}(x,t)+F_{2,\e}(x,t)\right)\om_{\e}(x,t)\,dx \\
&=\frac{1}{\Omega_{\e}}\int(x-p_{\e}(t))\cdot\left(F_{1,\e}(x,t)-F_{1,\e}(p_{\e}(t),t)\right)\om_{\e}(x,t)\,dx \\
& \,\,\,\,\,\,+ \frac{1}{\Omega_{\e}}\int(x-p_{\e}(t))\cdot F_{2,\e}(x,t)\om_{\e}(x,t)\,dx :=A_{1,\e}+A_{2,\e}.
\end{aligned}\]
Next by Lipschitz property of $F_{1,\e}$, we see that
$$
|A_{1,\e}| \le \frac{L}{\Omega_{\e}} \int |x-p_{\e}(t)|^2 \om_{\e}(x,t)\,dx=2LI_{\e}(t)
$$
and as a consequence of Jensen's inequality,
\[\begin{aligned}
|A_{2,\e}| &\le \|F_{2,\e}\|_{\infty}\int |x-p_{\e}(t)| \frac{\om_{\e}(x,t)}{\Omega_{\e}} \,dx \\
&\le \|F_{2,\e}\|_{\infty}\bigg(\int |x-p_{\e}(t)|^2 \frac{\om_{\e}(x,t)}{\Omega_{\e}} \,dx\bigg)^{\frac12} \\
&=\|F_{2,\e}\|_{\infty} \sqrt{2I_{\e}(t)}.
\end{aligned}\]
Thus we obtain
\[
\frac{dI_{\e}(t)}{dt}\le 2LI_{\e}(t)+\|F_{2,\e}\|_{\infty} \sqrt{2I_{\e}(t)},
\]
which implies
\[
\frac{d\sqrt{I_{\e}(t)}}{dt}\le L\sqrt{I_{\e}(t)}+\frac{1}{\sqrt{2}}\|F_{2,\e}\|_{\infty}
\]
and the bound for $I_{\e}(t)$ follows directly by Gr\"{o}nwall's inequality.
Next we compute $\frac{d}{dt}|p_{\e}(t)-P(t)|$. Due to \eqref{eq integration by parts1} and \eqref{eq section2 antisymmetry}, we get
\[\begin{aligned}
\frac{dp_{\e}(t)}{dt}&=\frac{1}{\Omega_{\e}}\int (u_{\e}(x,t)+F_{1,\e}(x,t)+F_{2,\e}(x,t))\om_{\e}(x,t)\,dx  \\
&=\frac{1}{\Omega_{\e}} \int F_{1,\e}(x,t)\om_{\e}(x,t)\,dx+\frac{1}{\Omega_{\e}} \int F_{2,\e}(x,t)\om_{\e}(x,t)\,dx.
\end{aligned}\]
So by definition of $P(t)$, we obtain
\[\begin{aligned}
\frac{d}{dt}\big(p_{\e}(t)-P(t)\big)&=\frac{1}{\Omega_{\e}} \int F_{1,\e}(x,t)\om_{\e}(x,t)\,dx-F_{1,\e}\big(P(t),t\big) +\frac{1}{\Omega_{\e}} \int F_{2,\e}(x,t)\om_{\e}(x,t)\,dx \\
&=F_{1,\e}\big(p_{\e}(t),t\big)-F_{1,\e}\big(P(t),t\big)+\frac{1}{\Omega_{\e}} \int \Big(F_{1,\e}(x,t)-F_{1,\e}\big(p_{\e}(t),t\big)\Big)\om_{\e}(x,t)\,dx \\
&\,\,\,+ \frac{1}{\Omega_{\e}} \int F_{2,\e}(x,t)\om_{\e}(x,t)\,dx \\
&:=T_1+T_2+T_3.
\end{aligned}\]
Thus,
\[
\frac{d}{dt}|p_{\e}(t)-P(t)| \le |\frac{d}{dt}(p_{\e}(t)-P(t))| \le |T_1|+|T_2|+|T_3|.
\]
For $T_1$, the Lipschitz continuity of  $F_{1,\e}$ yields
\[
|T_1| \le L |p_{\e}(t)-P(t)|.
\]
For $T_2$, as a consequence of Jensen's inequality and the Lipschitz property of $F_{1,\e}$, we obtain
\[\begin{aligned}
|T_2| \le& L \int |x-p_{\e}(t)|\frac{\om_{\e}(x,t)}{\Omega_{\e}} \,dx \\
\le& L \left(\int |x-p_{\e}(t)|^2 \frac{\om_{\e}(x,t)}{\Omega_{\e}} \,dx\right)^{\frac12}.
\end{aligned}\]
Then by the definition of $I_{\e}(t)$, we get
\[
|T_2| \le \sqrt{2}L \sqrt{I_{\e}(t)}.
\]
Thus, \eqref{eq section2 control of second momentum} implies
\[
|T_2| \le 2Le^{Lt} \left(\sqrt{I_{\e}(0)}+\frac{\|F_{2,\e}\|_{\infty}}{\sqrt{2}}\right)\int_0^t e^{-Ls} \,ds.
\]
For $T_3$, a direct calculation shows that
\[
|T_3|\le \|F_{2,\e}\|_{\infty}.
\]
Combining the above estimates, equation \eqref{eq section2 control of center of gravity} follows directly from Gr\"{o}nwall's inequality.
\end{proof}
\begin{remark}
Recall that the initial vorticity we considered satisfies $I_{\e}(0)\le A_{\e}^2$ and $|p_{\e}(0)-P(0)| \le A_{\e}$. So Theorem \ref{thm section2 main} tells us that the center of gravity of the vorticity remains concentrated near the point $P(t)$, where $P(t)$ can be viewed as an evolution of the point $P(0)$ under the velocity field $F_{1,\e}$.
\end{remark}
\begin{remark}
The case when $F_{2,\e}=0$ has already been considered in \cite{MaP}, so this Theorem shows that the presence of the perturbation term $F_{2,\e}$ will not change the dynamic of the Euler flow when $\|F_{2,\e}\|_{\infty}$ is small.
\end{remark}

Now we state our main Theorem in this section.
\begin{theorem} \label{thm section main main}
Assume $\om_{0,\e} \in L^{1} \cap L^{\infty}$, $\om_{0,\e}$ does not change sign, $$\text{supp}(\om_{0,\e}) \subset B_{\e}(p(0)) $$
  for some point $p(0)\in \R^2$,
  $$ \int \om_{0, \e} \,dx=\Omega_{\e} \to \Omega \neq 0.$$
and $\om_\e$ satisfies all the assumptions in Theorem \ref{thm section2 main}.
We assume also that $A_{\e} \to 0$ as $\e \to 0$ and there exist some $\gamma >0$ , $p_1>2$ such that $$\| \om_{0,\e}\|_{p_1} \le A_{\e}^{-\gamma},$$
where $A_{\e}$ is defined in section 1 and satisfies
$$ \max \left\{      \e,\|F_{2,\e}\|_{\infty}         \right\}\le CA_{\e}$$
for some $C>0$ independent of $\e$. Then for any $T < T^*$ and $a < \frac12$, there exists $\e_0 >0 $ such that for all $0 < \e \le \e_0$, we have
$$\text{supp}(\om_{\e}(\cdot,t))\subset B_{A_{\e}^{a}}(p_{\e}(t)).$$

\end{theorem}

\begin{remark}
Indeed, we do not need $\Omega_{\e}$ converges to $\Omega$ as $\e \to 0$. We will prove the Theorem under the assumption that $\frac{1}{C}< |\Omega_{\e}| <C$ for some $C>0$.
\end{remark}
Before proving the theorem, we give some technical lemma first. Lemma \ref{lemma control of the radial component of velocity}, \ref{lemma velocity out of singularity} and \ref{lemma radial component of velocity} are due to Marchioro and Pulvirenti \cite{MaP}. See also \cite{Jer} for details.
\begin{lemma}[\cite{Jer}, \cite{MaP}]\label{lemma control of the radial component of velocity}
Let $\om_{\e}:\R^2 \to [0,\infty)$ be a bounded, nonnegative function. Set
\[\begin{aligned}
\Omega_{\e}&=\int \om_{\e}\,dx, \\
p_{\e}&=\frac{1}{\Omega_{\e}} \int x\om_{\e}(x)\,dx
\end{aligned}\]
and define
$$
I(\om_{\e})=\frac{1}{\Omega_{\e}} \int|x-p_{\e}|^2 \om_{\e}(x)\,dx.
$$
Then for any $x\neq p_{\e}$ and any $0 <r <|x-p_{\e}|$, we have
$$
\bigg|\frac{x-p_{\e}}{|x-p_{\e}|}\cdot \int_{B_r(p_{\e})} \frac{(x-y)^{\perp}}{|x-y|^2}\om_{\e}(y) \,dy \bigg| \le \frac{C}{(|x-p_{\e}|-r)^2}\frac{I(\om_{\e})}{r}.
$$
\end{lemma}

\begin{lemma}[\cite{Jer}, \cite{MaP}]\label{lemma velocity out of singularity}
Assume $F$ is a Lipschitz continuous function with Lipschitz constant $L$ and $\om_{\e}$ satisfies the same assumptions in Lemma \ref{lemma control of the radial component of velocity}. Then for any $x \in \R^2$, we have
$$
\Big| F(x)-\frac{1}{\Omega_{\e}}\int F(y)\om_{\e}\,dy \Big|\le CL\Big(|x-p_{\e}|+\sqrt{I(\om_{\e})}\Big).
$$
\end{lemma}

\begin{lemma}[\cite{Jer}, \cite{MaP}]\label{lemma radial component of velocity}
Assume that $G \in W^{1,\infty}(\R^2,\R^2)$ satisfies
$$
G(x)=0 \text{ in $B_R(p_{\e})$}
$$
for some $R > 0$ and
$$G(x) \cdot (x-p_{\e})=0 \quad a.e. \,\,\,\,,$$
then
$$
\Big|\int_{\R^2\times \R^2}G(x) \cdot \frac{(x-y)^{\perp}}{|x-y|^2}\om_{\e}(x) \om_{\e}(y)\,dxdy\Big| \le C(\frac{\|G\|_{\infty}}{R^2}+\frac{\|\nabla G\|_{\infty}}{R^2})m(R)I(\om_{\e}),
$$
where
$$
m(R):=\int_{B^c_R(p_{\e})} |\om_{\e}(y)|\,dy.
$$
\end{lemma}

Next in order to control the velocity away from $p_{\e}(t)$, we proceed as in \cite{MaP} and \cite{Ser1}. Define a smooth cut-off function
\begin{equation*}
\chi_{R}(x)=
\begin{cases}
1 \,\,\,\,\, |x|\le R \\
0 \,\,\,\,\, |x|> 2R
\end{cases}\end{equation*}
which is nonnegative, radially decreasing and satisfies
$$
|\nabla \chi_{R}(x)|\le \frac{C}{R} \quad \text{and} \quad |D^2 \chi_R(x)|\le \frac{C}{R^2}
$$
for all $x\in \R^2$, then we obtain
\begin{lemma}\label{lemma section2 main}
Assume $F_{1,\e}$ is Lipschitz continuous with Lipschitz constant $L$. Define
\[\begin{aligned}
\mu_t(R)&=\frac{1}{\Omega_{\e}}\int\bigg(1-\chi_{R}\big(x-p_{\e}(t)\big)\bigg)\om_{\e}(x,t)\,dx, \\
m_t(r)&=\int_{B_r(p_{\e}(t))}|\om_{\e}(x,t)| \,dx.
\end{aligned}\]
Then
\begin{equation}\label{eq section2 **}
\frac{d}{dt} \mu_t(R)\le C\left(L+\sqrt{I_{\e}(t)}+\frac{I_{\e}(t)}{R^4}+\frac{\|F_{2,\e}\|_{\infty}}{R}\right)m_t(R).
\end{equation}
\end{lemma}
\begin{proof}
We will abbreviate $\chi_{R}$ as $\chi$ for simplicity. Then due to \eqref{eq integration by parts1} we obtain
\begin{equation*}\begin{aligned}
\frac{d}{dt}\mu_t(R)&=\frac{1}{\Omega_{\e}}\int \nabla \chi\big(x-p_{\e}(t)\big)\cdot \left(\frac{d}{dt}p_{\e}(t)-u_{\e}(x,t)-F_{1,\e}(x,t)-F_{2,\e}(x,t)\right)\,dx \\
&:=T_1+T_2+T_3,
\end{aligned}\end{equation*}
where
$$
T_1:=\frac{1}{\Omega_{\e}}\int \nabla \chi\big(x-p_{\e}(t)\big)\cdot \left(\int F_{1,\e}(y,t)\frac{\om_{\e}(y,t)}{\Omega_{\e}}\,dy-F_{1,\e}(x,t)\right)\,dx,
$$
$$
T_2:=-\frac{1}{\Omega_{\e}}\int \nabla \chi\big(x-p_{\e}(t)\big)\cdot \frac{(x-y)^{\perp}}{2\pi|x-y|^2}\om_{\e}(x,t)\om_{\e}(y,t)\,dx\,dy\,\,\,\,\,\,\,\,\,\,\,\,\,\,\,\,\,\,\,\,
$$
and
$$
T_3:=-\frac{1}{\Omega_{\e}}\int \nabla \chi\big(x-p_{\e}(t)\big)\cdot F_{2,\e}(x,t)\om_{\e}(x,t)\,dx.\,\,\,\,\,\,\,\,\,\,\,\,\,\,\,\,\,\,\,\,\,\,\,\,\,\,\,\,\,\,\,\,\,\,\,\,\,\,\,\,\,\,\,\,\,\,\,\,\,
$$
Note that $\nabla\chi\big(x-p_{\e}(t)\big)$ is supported in the annulus
$$
\Lambda_t :=\{x:R\le |x-p_{\e}(t)| \le 2R \}
$$
and recall that
$$|\nabla \chi| \le \frac{C}{R},$$
so
$$
|T_3| \le \frac{C}{R}\|F_{2,\e}\|_{\infty}\int_{B_R(p_{\e})^c} |\om_{\e}|\,dx \le \frac{C}{R}\|F_{2,\e}\|_{\infty} m_t(R)
$$

and we obtain
\[\begin{aligned}
|T_1| \le \frac{C}{R}\int_{\Lambda_t} \Bigg|F_{1,\e}(x,t)-\int F_{1,\e}(y,t) \frac{\om_{\e}(y,t)}{\Omega_{\e}}\,dy\Bigg|\,dx.
\end{aligned}\]
Thus, by Lemma \ref{lemma velocity out of singularity}
\[
|T_1|\le \frac{C}{R} \Big(RL+\sqrt{I_{\e}(t)}\Big)m_t(R).
\]
Finally, due to Lemma \ref{lemma radial component of velocity} we get
$$
|T_2|\le \frac{C}{R^4}m_t(R)I_{\e}(t).
$$
Combining the above estimates we get the desired conclusion.
\end{proof}
With the help of above lemmas, now we prove our main Theorem of this section.
\begin{proof} Proof of Theorem \ref{thm section main main}.

\emph{Step $1$:} ( Control of the vorticity away from $p_{\e}(t)$.) Fix $T >0$ and we see clearly from the definition of $\mu_t$ and $m_t$ that
$$\mu_t(R)\le m_t(R) \le \mu_t(R/2)$$
for any $t\in [0,T)$ and $R >0$.
Therefore, by \eqref{eq section2 control of second momentum}, \eqref{eq section2 **} and the assumptions on initial vorticity, we get
$$
\mu_t(R) \le \mu_0(R)+C_R\int_0^t \mu_{t_1}(R/2)\,dt_1,
$$
where
\[\begin{aligned}
C_R=&C\left(L+\sqrt{I_{\e}(t)}+\frac{I_{\e}(t)}{R^4}+\frac{\|F_{2,\e}\|_{\infty}}{R}\right)
\le& C(1+A_{\e}+\frac{A_{\e}^2}{R^4}+\frac{A_{\e}}{R}).
\end{aligned}\]
Note that $\mu_0(R)=0$ when $R\ge \sqrt{A_{\e}}$, so for all $R\ge \sqrt{A_{\e}}$ we have
$$\mu_t(R)\le \mu_0(R)+C\int_0^t \mu_{t_1}(R/2)\,dt_1=C\int_0^t \mu_{t_1}(R/2)\,dt_1.$$
Iterating $k$ times, we obtain
\[
\begin{aligned}
\mu_t(2^k \sqrt{A_{\e}}) \le& C\int_0^t \mu_{t_1}(2^{k-1}R)\,dt_1  \\
\le& C^k \int_0^t \cdot\cdot\cdot \int_0^{t_{k-1}} \mu_{t_k}(\sqrt{A_{\e}}) \,dt_{k}\cdot\cdot\cdot dt_1.
\end{aligned}\]
Since $\om_{\e}(\cdot,t)$ is a rearrangement of $\om_{0,\e}(\cdot)$, we get
$$\mu_t(R) \le \|\om_{\e}(\cdot,t)\|_1=\|\om_{0,\e}\|_1=|\Omega_{\e}|$$
for any $R>0$. Then as a consequence of Stirling's formula,
\[
\begin{aligned}
\mu_t(2^k \sqrt{A_{\e}}) \le& C^k \int_0^t \cdot\cdot\cdot \int_0^{t_{k-1}} \mu_{t_k}(\sqrt{A_{\e}}) \,dt_{k}\cdot\cdot\cdot dt_1 \\
\le& |\Omega_{\e}|\frac{C^k t^k}{k!}
\le |\Omega_{\e}|\frac{C^kT^ke^k}{k^{k+1/2}}
\le \frac{C^k}{k^{k+1/2}}.
\end{aligned}\]
Now we fix $\beta \in (a,\frac12)$ and choose $k=k(\e)$ such that
$$
\frac{A_{\e}^{\beta}}{2} \le 2^k \sqrt{A_{\e}} < A_{\e}^{\beta}.
$$
More precisely, we take
$$
k=\left\lfloor \frac{(\frac12-\beta)|\log A_{\e}|}{\log 2} \right\rfloor -1,
$$
where $\lfloor x \rfloor$ is the smallest integer greater than $x$. With such a choice of $k$, we then have
\[
\mu_t(A_{\e}^{\beta})\le \mu_t(2^k \sqrt{A_{\e}}) \le \frac{C^{|\log A_{\e}|}}{(\tilde{C}|\log A_{\e|})^{\tilde{C}|\log A_{\e}|}},
\]
where
\[
\tilde{C}=\frac{\frac12-\beta}{2\log 2} >0
\]
since $\beta<\frac12$. A direct calculation shows that when $\e_0=\e_0(C,\tilde{C},\gamma_1)$ is chosen small enough, then for any $\e \le \e_0$ we have
\[
\frac{C^{|\log\e|}}{(\tilde{C}|\log\e|)^{\tilde{C}|\log\e|}} \le \e^{\gamma_1}.
\]
 Thus we obtain
\begin{equation}\label{eq section2 control of outer mass}
\mu_t(A_{\e}^{\beta})\le A_{\e}^{\gamma_1}
\end{equation}
for all $\e \le \e_0(T,\beta,a,\gamma_1)$ small enough since $A_{\e} \to 0$ as $\e \to 0$.

 \emph{ Step $2$:} ( Control of the particle trajectory map $X(\alpha,t)$.) Now we fix $\alpha \in \R^2$ such that
 $$
 \alpha \in \text{supp}(\om_{0,\e}) \subset B_{\e}(p(0)).
 $$
As a result,
 $$|p(0)-p_{\e}(0)|=\Big|\frac{1}{\Omega_{\e}} \int \big(x-p(0)\big)\om_{0,\e}(x)\,dx\Big|  \le \e$$
 and hence
 $$
  \alpha \in \text{supp}(\om_{0,\e}) \subset B_{2 \e}(p_{\e}(0)).
 $$
Let $X(t):=X(\alpha,t)$ be the solution of
\[\begin{cases}
\frac{dX}{dt}=\left(u_{\e}+F_{1,\e}+F_{2,\e}\right)(X(t),t) \\
X(\alpha,0)=\alpha,
\end{cases}\]
where
\[
u_{\e}(x,t):=\frac{1}{2\pi} \int \frac{(x-y)^{\perp}}{|x-y|^2} \om_{\e}(y,t)\,dy.
\]
Define
\[
R(t):= \max \Big\{ |X(\alpha,t)-p_{\e}(t)| , 8A_{\e}^{\beta} \Big\}.
\]
We claim
\[
\frac{dR}{dt} \le CR(t)
\]
for almost every $t\in [0,T]$.
Since $R(t) \in W^{1,\infty}([0,T] \to \R)$, standard measure theory tells us
\[
\frac{d}{dt}R=0 \quad \text{a.e. on each level set of $R(t)$}.
\]
In particular $\frac{dR}{dt}=0$ on the level set $R(t)=8A_{\e}^{\beta}$. In order to get an upper bound of $R(t)$, it suffice to consider the time $t\in [0,T]$ when $R(t) > 8 A_{\e}^{\beta} \gg A_{\e}$. In such cases we have
\[
R(t)=|X(\alpha,t)-p_{\e}(t)|.
\]
Direct computation gives
\[\begin{aligned}
\frac{dR}{dt}&=\frac{X(t)-p_{\e}(t)}{|X(t)-p_{\e}(t)|}\cdot \left(u_{\e}\left(X(t),t\right)+F_{1,\e}(X(t),t)+F_{2,\e}(X(t),t)-\int \big(F_{1,\e}(y,t)+F_{2,\e}(y,t)\big)\om_{\e}(y,t) \,dy \right) \\
&:= T_1+T_2+T_3+T_4,
\end{aligned}\]
where
\[\begin{aligned}
T_1&=\frac{X(t)-p_{\e}(t)}{|X(t)-p_{\e}(t)|}\cdot u_{\e}(X(t),t),\\
T_2&=\frac{X(t)-p_{\e}(t)}{|X(t)-p_{\e}(t)|}\cdot \left(F_{1,\e}\big(X(t),t\big)-\int F_{1,\e}(y,t)\om_{\e}(y,t) \,dy \right), \\
T_3&=\frac{X(t)-p_{\e}(t)}{|X(t)-p_{\e}(t)|}\cdot F_{2,\e}\big(X(t),t\big), \\
T_4&=\frac{X(t)-p_{\e}(t)}{|X(t)-p_{\e}(t)|}\cdot \int F_{2,\e}(y,t)\om_{\e}(y,t) \,dy .
\end{aligned}\]
We first estimate
\[\begin{aligned}
T_1=&\frac{1}{2\pi}\int\frac{X(t)-p_{\e}(t)}{|X(t)-p_{\e}(t)|}\cdot \frac{(x-y)^{\perp}}{|x-y|^2}\om_{\e}(y,t) \,dy \\
=&\frac{1}{2\pi}\int_{B_{\frac{R(t)}{2}}(p_{\e}(t))} \frac{X(t)-p_{\e}(t)}{|X(t)-p_{\e}(t)|}\cdot \frac{(x-y)^{\perp}}{|x-y|^2}\om_{\e}(y,t) \,dy \\ &+\frac{1}{2\pi}\int_{B^c_{\frac{R(t)}{2}}(p_{\e}(t))} \frac{X(t)-p_{\e}(t)}{|X(t)-p_{\e}(t)|}\cdot \frac{(x-y)^{\perp}}{|x-y|^2}\om_{\e}(y,t) \,dy \\
:&=T_{11}+T_{12}.
\end{aligned}\]
Recall that $R(t)=|X(t)-p_{\e}(t)| \ge 8A_{\e}^{\beta}$ and $I_{\e}(t) \le CA_{\e}^2$, so we obtain from Lemma \ref{lemma control of the radial component of velocity} that
\[
|T_{11}| \le \frac{C}{R(t)^3} I_{\e}(t) \le C A_{\e}^{2-3\beta}\le A_{\e}^{\beta}
\]
since $\beta < \frac12$ and $A_{\e} \to 0$ as $\e \to 0$. Next we consider $T_{12}$, by Lemma \ref{lemma control of the velocity} we have
  \[\begin{aligned}
  |T_{12}|\le& \frac{1}{2\pi} \int \left|\frac{1}{|x-y|}\om_{\e}(y,t)\mathbf{1}_{{B^c_{\frac{R(t)}{2}}(p_{\e}(t))}}\right|\,dy \\
  &\le C\|\om_{\e}(y,t)\mathbf{1}_{{B^c_{\frac{R(t)}{2}}(p_{\e}(t))}}\|_1^{\frac{p_1-2}{2p_1-2}} \|\om_{\e}(y,t)\mathbf{1}_{{B^c_{\frac{R(t)}{2}}(p_{\e}(t))}}\|_{p_1}^{\frac{p_1}{2p_1-2}},
  \end{aligned}\]
  where $\mathbf{1}_{\Omega}$ is the characteristic function on $\Omega$ and $p_1>2$ is the number we defined in Theorem \ref{thm section main main}. Therefore, by definition of $p_1$, we have
\[
\|\om_{\e}(y,t)\mathbf{1}_{{B^c_{\frac{R(t)}{2}}(p_{\e}(t))}}\|_{p_1} \le \| \om_{\e} \|_{p_1} = \|\om_{0,\e}\|_{p_1} \le A_{\e}^{-\gamma}.
\]
 Then by \eqref{eq section2 control of outer mass}, it follows that
\[
\|\om_{\e}(y,t)\mathbf{1}_{{B^c_{\frac{R(t)}{2}}(p_{\e}(t))}}\|_1=m_{t}(\frac{R(t)}{2})\le m_t(4A_{\e}^{\beta})\le \mu_t(A_{\e}^{\beta}) \le A_{\e}^{\gamma_1}.
\]
As a result, if we choose $\gamma_1=\gamma_1(p_1, \beta, \gamma)$ large enough, then
\[\begin{aligned}
|T_{12}| \le& A_{\e}^{\frac{\gamma_1(p_1-2)-\gamma p_1}{2p_1-2}} \le A_{\e}^{\beta} \le R(t).
\end{aligned}\]
 For $T_2$, since $R(t) > 8 A_{\e}^{\beta} \gg A_{\e}$, Lemma \ref{lemma velocity out of singularity} yields
\[
|T_2| \le  C(|X(t)-p_{\e}(t)|+A_{\e}) \le CR(t).
\]

For $T_3,T_4$, the fact $\|F_{2,\e}\|_{\infty} \le A_{\e}$ implies
\[
|T_3|+|T_4| \le CA_{\e} \le R(t),
\]

which shows $\frac{dR}{dt}\le CR(t)$. Together with the fact that $R(0)=8A_{\e}^{\beta}$,  Gr\"{o}nwall's inequality gives
\[
R(t)\le CR(0) \le CA_{\e}^{\beta}.
\]
Now recall that $a<\beta < \frac12$, so for all $\e \le \e_0$, choosing $\e_0=\e_0(a,\beta)$ small enough, then we have
\[
R(t) \le A_{\e}^{a}.
\]

\emph{Step $3$:} ( Final arguments.) Note that $\om_{\e}(X(\alpha,t),t)=\om_{0,\e}(\alpha)$ by vorticity transport formula, so
\[
\text{supp}(\om_{\e}(\cdot,t))=\{X(\alpha,t): \alpha \in \text{supp}(\om_{0,\e})\}.
\]
Then by Step $2$, we see that
$$|X(\alpha,t)-p_{\e}(t)|\le A_{\e}^a$$
for all $\alpha \in \text{supp}(\om_{0,\e})$. As a result,
\[
\text{supp}(\om_{\e}(\cdot,t))\subset B_{A_{\e}^a}(p_{\e}(t)),
\]
which completes the proof of the Theorem.
\end{proof}
\begin{remark}\label{remark section2 not lipschitz}
Indeed, we do not need $F_{1,\e}$ to be Lipschitz continuous in $\R^2$. Theorem \ref{thm section main main} remains valid if we only assume $F_{1,\e}(\cdot,t)$ to be Lipschitz continuous in some convex neighborhood that contains $P(t)$, $p_{\e}(t)$ and $\text{supp}(\om_{\e}(\cdot,t))$.
\end{remark}

\section{Proof of the main Theorem.}
\subsection{Proof of Theorem \ref{thm main}.}
Together with the prior estimates we obtained in Section 2, now we prove Theorem \ref{thm main} by a bootstrap argument.
\begin{proof}{Proof of Theorem \ref{thm main}}  .
From Lemma \ref{lemma euler decomposition}, the solution $\om_{\e}(x,t)$ of equation \eqref{eq perturb} admits a decomposition $$\om_{\e}=\sum_{m=1}^M \om_{m,\e}+\om_{p,\e},$$ where $\om_{m,\e}(\cdot,t)$ is a rearrangement of $\om_{0m,\e}(\cdot)$ and  satisfies
\begin{equation*}\begin{cases}\begin{aligned}
\partial_t \om_{m,\e}+(u_{m,\e}+F_{1,m,\e}+F_{2,m,\e}) \cdot \nabla \om_{m,\e}=0 \\
\om_{m,\e}(\cdot,0)=\om_{0m,\e}
\end{aligned}\end{cases}\end{equation*}

in weak sense with $F_{1,m,\e}=\sum_{l\neq m}u_{l,\e}$ and $F_{2,m,\e}=u_{p,\e}$. Now in order to apply Theorem \ref{thm section main main} to $\om_{m,\e}$, we need to check that \begin{enumerate}
\item $F_{1,m,\e}$ is Lipschitz continuous and the Lipschitz constant is independent of $\e$ and the time $t$.
\item $\|F_{2,m,\e}\|_{\infty}\to 0 \quad \text{as $\e \to 0$}.$
\item The vorticity transport formula holds:
$$\om_{m,\e}(X(\alpha,t),t)=\om_{0m,\e}(\alpha),$$
where $X(\alpha,t)$ satisfies
\[\begin{cases}
\frac{dX(\alpha,t)}{dt}=(u_{m,\e}+F_{1,m,\e}+F_{2,m,\e})(X(\alpha,t),t) \\
X(\alpha,0)=\alpha.
\end{cases}\]
\item For any time $t$, the map $\alpha \to X(\alpha,t)$ is measure preserving.
\item $\om_{0m,\e} \in L^1 \cap L^{\infty}$ and does not change sign.
\item $\frac{1}{C} < \int |\om_{0m,\e}|\,dx <C$ for some $C>0$.
\item There exist some $\gamma>0,\, p_1>2$ such that $\|\om_{0m,\e}\|_{p_1}\le A_{\e}^{-\gamma}$.
\item $A_{\e} \to 0$ as $\e \to 0$ and $A_{\e}\ge\max \{ \e,\|F_{2,m,\e}\|_{\infty} \}$.

\end{enumerate}
Condition (v),(vi),(vii) follows directly from our assumptions of $\om_{0m,\e}$ in Theorem \ref{thm main}. Condition (ii),(iii),(iv),(vii),(viii) followed by Lemma \ref{lemma euler decomposition} and our assumptions on $\om_{0p,\e}$ in Theorem \ref{thm main}. Thus it remains to verify condition (i).
Define

\[\begin{aligned}
\delta:=&\min \{|p_m(t)-p_l(t)| : 0\le t \le T, m\neq l \} >0, \\
p_{m,\e}(t):=&\frac{1}{\gamma_{m,\e}}\int x\om_{m,\e}(x,t)\,dx, \\
I_{m,\e}(t):=&\frac{1}{2\gamma_{m,\e}} \int |x-p_{m,\e}(t)|^2\om_{m,\e}(x,t)\,dx
\end{aligned}\]
and let $P_{m,\e}(t)$ be the solution of the following equation
\[\begin{cases}
\frac{d}{dt}P_{m,\e}=F_{1m,\e}(P_{m,\e}(t),t) \\
P_{m,\e}(0)=p_m(0).
\end{cases}\]
Set \[\begin{aligned}T_{\e}:=\sup \Big\{&t \in [0,T]:\text{supp}(\om_{m,\e}(\cdot,s))\subset B_{\frac{\delta}{16}}(p_{m}(s)), \\ &|P_{m,\e}(s)-p_m(s)|+|p_{m,\e}(s)-p_m(s)| \le \frac{\delta}{16} \quad \text{for $m=1, ..., M$ and $0 \le s \le t$} \Big\}. \end{aligned}\]
Then by Lemma \ref{lemma lipschitz property} we see that
\[
|\nabla F_{1,\e}(x,t)|\le \frac{C}{\delta^2} \sum_{l\neq m} \|\om_{0m,\e}\|_{1} \le C.
\]
 for all $t \le T_{\e}$ and $x \in B_{\frac{\delta}{2}}(p_m(t))$. Next from the definition of $T_{\e}$, we find that $B_{\frac{\delta}{2}}(p_m(t))$ contains $p_{m,\e}(t),P_{m,\e}(t)$ and $\text{supp}(\om_{m,\e}(\cdot,t))$ for all $t \le T_{\e}$. Thus by Remark \ref{remark section2 not lipschitz}, Theorem \ref{thm section main main} can be applied to $\om_{m,\e}$ on the time interval $[0,T_{\e}]$.

First, it follows from Theorem \ref{thm section2 main} that
\[
I_{m,\e}(t) \le CA_{\e}^2
\]
and
\[
|p_{m,\e}(t)-P_{m,\e}(t)|\le CA_{\e}.
\]
Then Theorem \ref{thm section main main} guarantees that for some  $a <b<\frac12$,
\[\begin{aligned}
\text{supp}(\om_{m,\e}(\cdot,t))\subset& B_{ A_{\e}^b}(p_{m,\e}(t))\subset B_{A_{\e}^b+CA_{\e}}(P_{m,\e}(t))\subset B_{2A_{\e}^b}(P_{m,\e}(t))
\end{aligned}\]
for all $0< \e \le \e_0$ once we choose $\e_0=\e_0(b)$ small enough. Define
\[
G(t)=\sum_{m=1}^M |P_{m,\e}(t)-p_m(t)|,
\]
we claim that
\begin{equation}\label{eq section3 *}\frac{d}{dt}G \le C(G(t)+A_{\e}^b).\end{equation}
Once the claim is proved, then Gr\"{o}nwall's inequality gives
\[
|P_{m,\e}(t)-p_m(t)|\le G(t) \le CA_{\e}^b,
\]
which implies
\[
|p_{m,\e}(t)-p_m(t)|\le |P_{m,\e}(t)-p_{m,\e}(t)|+|P_{m,\e}(t)-p_m(t)|\le CA_{\e}+CA_{\e}^b.
\]
Thus, we obtain
\begin{equation}\label{eq section3 main}
\text{supp}(\om_{m,\e}(\cdot,t)) \subset B_{CA_{\e}^{b}}(p_m(t)) \subset B_{A_{\e}^a}(p_m(t))
\end{equation}
for $0< \e \le \e_0$ if we choose $\e_0=\e_0(a,b)$ small enough. As a result, for all $t \in [0,T_{\e}]$, $0<\e \le \e_0$ and $m=1, ..., M$, we have
\[
\text{supp}(\om_{m,\e}(\cdot,t))\subset B_{\frac{\delta}{32}}(p_{m}(t))
\]
and
\[
|P_{m,\e}(t)-p_m(t)|+|p_{m,\e}(t)-p_m(t)| \le \frac{\delta}{32},
\]
which contradicts the definition of $T_{\e}$ if $T_{\e}<T$. So $T_{\e}=T$ and \eqref{eq section3 main} holds for $t \in [0,T]$. Therefore, it suffice to prove \eqref{eq section3 *}. In fact, let $K(z)=\frac{z^{\perp}}{2\pi |z|^2}$ be the Biot-Savart kernel, since $p_m(t)$ satisfies the Helmholtz equation \eqref{eq helmholtz}, it holds that
\[\begin{aligned}
\frac{d}{dt}G\le& \sum_{m=1}^M \Big|\frac{d}{dt}(P_{m,\e}-p_{m})\Big| \\
=& \sum_{m=1}^M \Big|F_{1m,\e}\big(P_{m,\e}(t),t\big)-\frac{d}{dt}p_m(t)\Big| \\
=& \sum_{m=1}^M \Big|F_{1m,\e}\big(P_{m,\e}(t),t\big)-\sum_{l \neq m}\gamma_l K\big(p_m(t)-p_l(t)\big)\Big| \\
\le& \sum_{m=1}^M \Big| F_{1m,\e}\big(P_{m,\e}(t),t\big)-F_{1m,\e}\big(p_m(t),t\big) \Big| \\
&+ \sum_{m=1}^M\Big|F_{1m,\e}\big(p_{m}(t),t\big)-\sum_{l \neq m}\gamma_{l,\e} K\big(p_m(t)-p_l(t)\big)\Big| \\
&+ \sum_{m=1}^M\Big|\sum_{l \neq m}\gamma_{l,\e} K\big(p_m(t)-p_l(t)\big)-\sum_{l \neq m}\gamma_{l} K\big(p_m(t)-p_l(t)\big)\Big| \\
:=& T_1+T_2+T_3.
\end{aligned}\]
 Recall that
$$
P_{m,\e}(t)\subset B_{CA_{\e}}(p_{m,\e}(t)) \subset B_{\frac{\delta}{3}}(p_m(t))
$$
and $|\nabla F_{1m,\e}| \le C$ in such region, so it follows that $|T_1| \le C G(t)$.
Then by definition of $F_{1m,\e}$, we get
$$F_{1m,\e}=\sum_{l\neq m} u_{l,\e}=\sum_{l \neq m}\int \gamma_{l,\e} K(x-y)\frac{\om_{l,\e}(y,t)}{\gamma_{l,\e}}\,dy.$$
Therefore,
\[\begin{aligned}
|T_2|\le& C\sum_{m=1}^M\sum_{l\neq m} \int \bigg|K\bigg(p_m(t)-y\bigg)-K\bigg(p_m(t)-p_l(t)\bigg)\bigg|\,\frac{\om_{l,\e}(y,t)}{\gamma_{l,\e}}\,dy \\
=& C\sum_{m=1}^M\sum_{l\neq m} \int_{B_{\frac{\delta}{4}}(p_l(t))} \bigg|K\bigg(p_m(t)-y\bigg)-K\bigg(p_m(t)-p_l(t)\bigg)\bigg|\,\frac{\om_{l,\e}(y,t)}{\gamma_{l,\e}}\,dy.
\end{aligned}\]
Since $|p_m(t)-p_l(t)| \ge \delta$ when $m \neq l$ and $|\nabla K(z)| \le \frac{C}{\delta^2} $ when $|z| \ge \frac{\delta}{16}$, by mean value Theorem we obtain
\[\begin{aligned}
|T_2| \le& C \sum _{l=1}^M \int \Big|\big(y-p_l(t)\big)\om_{l,\e}(y,t)\Big|\,dy \\
\le& \sum_{l=1}^M \int \Big|\big(y-P_{l,\e}(t)\big)\om_{l,\e}(y,t)\Big|\,dy+\sum_{l=1}^M \int \Big|\big(P_{l,\e}(t)-p_l(t)\big)\om_{l,\e}(y,t)\Big|\,dy \\
:=& T_{21}+T_{22}.
\end{aligned}\]
For $T_{22}$, it is obvious that $T_{22} \le CG(t)$.
For $T_{21}$, the fact
$$\text{supp}(\om_{l,\e}(\cdot,t))\subset B_{2A_{\e}^b}(P_{l,\e}(t))$$
implies $$T_{21}\le CA_{\e}^{b}$$ and hence $$ |T_2| \le C(G(t)+A_{\e}^b).$$
Finally, as a consequence of the assumptions $|\gamma_{m,\e}-\gamma_m| \le CA_{\e}^{\frac12}$ and recall that $|p_m(t)-p_l(t)|\ge \delta$ when $m\neq l$, we get
\[
T_3 \le CA_{\e}^{b}
\]
since $b< \frac12$. Combining the above estimates we obtain
\[
\frac{dG}{dt} \le C(G(t)+A_{\e}^b)
\]
and the proof of the Theorem is complete.
\end{proof}
\subsection{Proof of Theorem \ref{thm concentration of mass}.}
Now we consider the initial vorticity which has the form
\begin{equation*}\om_{0,\e}(x)=\sum_{m=1}^M \frac{\gamma_m}{\e^2} \eta(\frac{x-p_m^0}{\e}),\end{equation*}
where $\eta$ is a nonnegative function satisfies
\begin{equation}\label{eq ??}\begin{aligned}
\int \eta(x)\,dx=1 \quad \text{and} \quad
\eta(x) \le \frac{C}{1+|y|^{2+\sigma}} \quad \text{for $\sigma>0$}.
\end{aligned}\end{equation}
We will show that as a consequence of Theorem \ref{thm main}, the solution $\om_{\e}(\cdot,t)$ remains concentrated near $M$ points $p_1(t), ..., p_M(t)$.

\begin{proposition}\label{prop main1}(Decomposition of the initial vorticity.)
Let $p_0$ be a point in $\R^2$ and $\om_{\e}(x)=\frac{1}{\e^2}\eta(\frac{x-p_0}{\e})$. Then $\om_{\e}$ admits a decomposition
\begin{equation*}
\om_{\e}(x)=\om_{1,\e}(x)+\om_{p,\e}(x),
\end{equation*}
where
\begin{equation}\label{eq section3 1}
\text{supp}(\om_{1,\e}) \subset B_{\e^{\frac{\sigma}{\sigma+2}}}(p_0) ,
\end{equation}
\begin{equation}\label{eq section3 2}
|1-\int \om_{1,\e}(x)\,dx| \le C\e^{\frac{\sigma^2}{\sigma+2}}
\end{equation}
and for any $q>2$, we have
\begin{equation}\label{eq section3 3}
\|\om_{p,\e}\|_q^{\frac{q}{2q-2}}\| \om_{p,\e}\|_1^{\frac{q-2}{2q-2}} \lesssim_q \e^{\frac{\sigma^2}{\sigma+2}}.
\end{equation}
\end{proposition}
\begin{proof}
We may assume with lose of generality that $p_0=(0,0)$, otherwise we replace $\eta(x)$ by $\tilde{\eta}(x):=\eta(x-p_0)$. For any $0<\beta <1$, we can decomposite $\om_\e$ as
\[\begin{aligned}
\frac{1}{\e^2}\eta(\frac{x}{\e})=\frac{1}{\e^2}\eta(\frac{x}{\e})\chi_{|x|\le \e^{1-\beta}}+\frac{1}{\e^2}\eta(\frac{x}{\e})\chi_{|x| > \e^{1-\beta}} :=\om_{1,\e}+\om_{p,\e}.
\end{aligned}\]
Define
\[
\gamma_{\e}=\int \om_{1,\e}(x)\,dx.
\]
Then from \eqref{eq ??} we have
\[\begin{aligned}
\gamma_{\e}=& \frac{1}{\e^2} \int_{|x|\le \e^{1-\beta}} \eta(\frac{x}{\e}) \,dx =1-\int_{|y|>\e^{-\beta}}\eta(y)\,dy.
\end{aligned}\]
Since
$|\eta(y)| \lesssim |y|^{-2-\sigma}$
for $|y|$ large, we obtain
\[
\int_{|y|>\e^{-\beta}} \eta(y)\,dy \le C \e^{\beta \sigma}
\]
and
\begin{equation}\label{eq section3 2.1}
|\gamma_{\e}-1|\le C\e^{\beta \sigma}.
\end{equation}
For $\om_{p,\e}$, we compute that
\[\begin{aligned}
\|\om_{p,\e}\|_q=&\frac{1}{\e^2} \int_{|x|>\e^{1-\beta}} \eta^q(\frac{x}{\e})\,dx^{\frac1q} \\
=& \e^{-2+\frac2q}\int_{|y|>\e^{-\beta}}\eta^q(y)\,dy^{\frac1q} \\
\lesssim& \e^{\beta \sigma} \e^{(1-\beta)(-2+\frac2q)}
\end{aligned}\]
since $|\eta(y)| \lesssim |y|^{-2-\sigma}$ and $q>2$.

Thus,
\begin{equation*}\begin{aligned}
\|\om_{p,\e}\|_q^{\frac{q}{2q-2}}\| \om_{p,\e}\|_1^{\frac{q-2}{2q-2}} \lesssim& \e^{\frac{\beta \sigma(q-2)}{2q-2}}\e^{\frac{q\beta\sigma+(1-\beta)(-2q+2)}{2q-2}}
=\e^{\beta\sigma+\beta-1}.
\end{aligned}\end{equation*}
To ensure that this quantity is small, we need an assumption that
\[
\beta\sigma+\beta-1>0,
\]
which means
\[
1>\beta >\frac{1}{\sigma+1} \quad \text{or equivalently} \quad \frac{\sigma}{1+\sigma}>1-\beta>0.
\]
Then we define
\begin{equation}\label{eq section3 A}
A_{\e}:= \max\{ \e^{1-\beta},\e^{\beta\sigma+\beta-1}\}.
\end{equation}
To minimize the quantity $A_{\e}$, we choose $\beta=\beta(\sigma)=\frac{2}{\sigma+2}$ which satisfies $1>\beta >\frac{1}{\sigma+1}$. With such choice of $\beta$, we get $A_{\e}=\e^{\frac{\sigma^2}{\sigma+2}}$ and equation \eqref{eq section3 3} follows directly. Meanwhile,
\[
\text{supp}(\om_{1,\e}) \subset B_{\e^{1-\beta}}= B_{\e^{\frac{\sigma}{\sigma+2}}} ,
\]
which is exactly \eqref{eq section3 1}. Inequality \eqref{eq section3 2} follows directly from \eqref{eq section3 2.1} and our choice of $\beta(\sigma)$.
\end{proof}
\begin{proof}{Proof of Theorem \ref{thm concentration of mass} (i).} According to Proposition \ref{prop main1}, $\om_{0,\e}$ admits a decomposition:
\begin{equation*}
\om_{0,\e}(x)=\sum_{m=1}^M \om_{0m,\e}(x)+\om_{0p,\e}(x),
\end{equation*}
where $\om_{0m,\e}$ does not change sign,
\[
\text{supp}(\om_{0m,\e})\subset B_{\e^{\frac{\sigma}{\sigma+2}}}(p_m^0)
\]
and
\[
|\gamma_m-\int \om_{0m,\e}(x)\,dx| \le C\e^{\frac{\sigma^2}{\sigma+2}}.
\]
Moreover, $\om_{0p,\e}$ satisfies
\[
\|\om_{0p,\e}\|_q^{\frac{q}{2q-2}}\| \om_{0p,\e}\|_1^{\frac{q-2}{2q-2}} \le C \e^{\frac{\sigma^2}{\sigma+2}}
\]
for $q=4$ and
\[
\|\om_{0p,\e}\|_1 \le C\e^{\frac{\sigma^2}{\sigma+2}}.
\]
Set $\tilde{\e}=\e^{\frac{\sigma}{\sigma+2}}$ and define
\[
W_{0,\tilde{\e}}(x):=\sum_{m=1}^M W_{0m,\tilde{\e}}(x)+W_{0p,\tilde{\e}}(x),
\]
where
\[
W_{0m,\tilde{\e}}(x):=\om_{0m,\e}(x) \quad \text{and} \quad W_{0p,\tilde{\e}}(x):=\om_{0p,\e}(x).
\]
Let $A_{\tilde{\e}},\gamma_{m,\tilde{\e}}$ be the quantity that we defined in section 1 with $\e$ replaced by $\tilde{\e}$, then we get $A_{\tilde{\e}}\le CA_{\e}$, where $A_{\e}=\e^{\frac{\sigma}{\sigma+2}}$ is defined in \eqref{eq section3 A}. Since $A_{\tilde{\e}} \to 0$ as $\e \to 0$, we have
\[
|\gamma_m-\gamma_{m,\tilde{\e}}|\le CA_{\tilde{\e}} \le CA_{\tilde{\e}}^{\frac12}.
\]
 A directly calculation shows that $\|W_{0m,\tilde{\e}}\|_{\infty}\lesssim \e^{-r_1}$ and $A_{\tilde{\e}}\approx \e^{r_2}$ for some $r_1,r_2>0$. So there exists $\gamma_1>0$ such that $\|W_{0m,\tilde{\e}}\|_{\infty} \le A_{\tilde{\e}}^{-\gamma_1}$. Thus, $W_{0,\tilde{\e}}$ satisfies all the assumptions we made in Theorem \ref{thm main}. Therefore, for any $a_1<\frac{1}{2}$, $\om_{m,\e}(\cdot,t)=W_{m,\tilde{\e}}(\cdot,t)$ is supported in $B_{(A_{\tilde{\e}})^{a_1}}(p_m(t))=B_{\e^{\frac{a_1\sigma}{\sigma+2}}}(p_m(t))$ when $\tilde{\e} \le \tilde{\e}_0(\sigma,a_1,T)$ and $0\le t \le T$. Moreover,
\[
\|\om_{p,\e}(\cdot,t)\|_1=\|W_{p,\tilde{\e}}(\cdot,t)\|_1=\|W_{0p,\tilde{\e}}\|_1\le C\e^{\frac{\sigma^2}{\sigma+2}}.
\]
Now for any $a< \frac{\sigma}{2(\sigma+2)}$, we take $a_1=\frac{a(\sigma+2)}{\sigma}<\frac12$ and $\e_0=\tilde{\e}_0^{\frac{\sigma+2}{\sigma}}$. Then for all $\e \le \e_0$, the solution $\om_{\e}(\cdot,t)$ admits a decomposition
\[
\om_{\e}(\cdot,t)=\sum_{m=1}^M \om_{m,\e}(\cdot,t)+\om_{p,\e}(\cdot,t),
\]
where $\sum_{m=1}^M \om_{m,\e}(\cdot,t)$ is supported in
\[
B(t):= \bigcup_{m=1}^M B_{\e^a}(p_m(t))
\]
and
\[
\|\om_{p,\e}(\cdot,t)\|_1 \le C\e^{\frac{\sigma^2}{\sigma+2}}.
\]
As a result,
\[
\int_{B(t)^c}|\om_{\e}(x,t)|\,dx \le C\e^{\frac{\sigma^2}{\sigma+2}}
\]
and the proof completes.

\end{proof}

\begin{proof}{Proof of Theorem \ref{thm concentration of mass} (ii).}
By the similar argument  in Proposition \ref{prop main1}, we have a decomposition
\[\begin{aligned}
\frac{1}{\e^2}\eta(\frac{x}{\e})&=\frac{1}{\e^2}\eta(\frac{x}{\e})\chi_{|x|\le \e^{1-\beta}}+\frac{1}{\e^2}\eta(\frac{x}{\e})\chi_{|x| > \e^{1-\beta}}\\
&:=\om_{1,\e}+\om_{p,\e}
\end{aligned}\]
with
\[
\|\om_{p,\e}\|_{1} \lesssim \e^{\beta\sigma}.
\]
For any $\frac{\sigma}{\sigma+1}<\gamma_1<\sigma$, we take $\beta=\beta(\sigma):=\frac{\gamma_1}{\sigma}$, then we have
\[
\|\om_{p,\e}\|_{1} \lesssim \e^{\gamma_1}.
\]
We also note that $\om_{1,\e}$ is supported in a ball of radius $A_{\e}$, where
\begin{equation*}
A_{\e}:= \max\{ \e^{1-\beta},\e^{\beta\sigma+\beta-1}\}=\e^{2a_0}
\end{equation*}
and \[
a_0=a_0(\gamma_1,\sigma):= \frac{1}{2}\min \left\{1-\frac{\gamma_1}{\sigma},\gamma_1+\frac{\gamma_1}{\sigma}-1 \right\}
\] is defined in Theorem \ref{thm concentration of mass}. Finally, by a similar argument as in the proof of Theorem \ref{thm concentration of mass} (i), we get the desired conclusion.
\end{proof}

\section{Long time dynamic for point vortices in Euler flows}
We will assume throughout this section that $|p_m(t)-p_l(t)|\ge \delta>0$ for all $t \in [0,\infty)$ when $m\neq l$.
\subsection{A single vortex with a background flow}
We first prove a similar result to Theorem \ref{thm section main main}.
\begin{theorem} \label{thm section4 main}
Assume $\om_{0,\e} \in L^{1} \cap L^{\infty}$, $\om_{0,\e}$ does not change sign, $$\text{supp}(\om_{0,\e}) \subset B_{\e}(p(0)) $$
for some point $p(0)\in \R^2$ and $\om_\e$ satisfies \eqref{eq single vortex}.
Let the total vorticity $\Omega_{\e}$ be
$$ \Omega_{\e}=\int \om_{0, \e} \,dx.$$
We also assume that $A_{\e} \to 0$ as $\e \to 0$ and there exist some $\gamma >0$ , $p_1>2$ such that $$\| \om_{0,\e}\|_{p_1} \le A_{\e}^{-\gamma}$$
with
$$\max \left\{      \e,\|F_{2,\e}\|_{\infty}         \right\} \le A_{\e}.$$

Then for any any $a < \frac12$, there exist $\e_0 >0 $ and $c_0 >0$ such that for all $0 < \e \le \e_0$ and $t \le c_0 |\log A_{\e}|$, we have
$$\text{supp}(\om_{\e}(\cdot,t))\subset B_{A_{\e}^{a}}(p_{\e}(t)).$$

\end{theorem}

\begin{proof}
With the same definition as in Section $2$, we see from Theorem \ref{thm section2 main} that
\[
I_{\e}(t) \le 2e^{2Lt}\left( I_{\e}(0)+\frac{\|F_{2,\e}\|_{\infty}^2}{2}\int_0^t e^{-Ls}\,ds^2\right)
\]
and
\[\begin{aligned}
|p_{\e}(t)-p(t)|\le e^{Lt} \Bigg( |p_{\e}(0)-p(0)|&+2L(\sqrt{I_{\e}(0)}+\frac{\|F_{2,\e}\|_{\infty}}{\sqrt{2}L})&\int_0^t \int_0^r e^{-Ls }\,ds\,dr \\&+\|F_{2,\e}\|_{\infty}\int_0^t e^{-Ls}\,ds \Bigg).
\end{aligned}\]
So for arbitrary $\sigma>0$, there exists $c_0=c_0(\sigma)$ small enough, such that for all $t \le c_0|\log A_{\e}|$ we have
\[
I_{\e}(t)\le A_{\e}^{2-2\sigma} \]
and
\[
|p_{\e}(t)-p(t)| \le A_{\e}^{1-\sigma}.
\]
Thus, Lemma \ref{lemma section2 main} yields
\[
\mu_t(R) \le \mu_0(R)+C\int_0^{t}\mu_{t_1}(R/2)\,dt_1=C\int_0^{t}\mu_{t_1}(R/2)\,dt_1
\]
for any $R\ge A_{\e}^{\frac{1-\sigma}{2}}$. Iterating $k$ times and applying Stirling's formula, we get
\[
\begin{aligned}
\mu_t(2^k A_{\e}^{\frac{1-\sigma}{2}}) \le& C^k \int_0^t \cdot\cdot\cdot \int_0^{t_{k-1}} \mu_{t_k}(A_{\e}^{\frac{1-\sigma}{2}}) \,dt_{k}\cdot\cdot\cdot dt_1 \\
\le& |\Omega_{\e}|\frac{C^k t^k}{k!}
\le \frac{|c_0C \log A_{\e}|^k}{k^{k+1/2}}
\le \frac{|c_0 \log A_{\e}|^k}{k^{k+1/2}}.
\end{aligned}\]

Now fix $\beta \in (a,\frac12)$, assume $\sigma$ is small enough such that $a<\beta(1-\sigma)$. Then we can choose $k=k(\e)$ such that
\[
\frac{A_{\e}^{\beta(1-\sigma)}}{2} \le 2^k A_{\e}^{\frac{1-\sigma}{2}} \le  A_{\e}^{\beta(1-\sigma)}.
\]
More precisely, we take
$$
k=\left\lfloor \frac{(1-\sigma)(\frac12-\beta)|\log A_{\e}|}{\log 2} \right\rfloor -1.
$$
We claim that for any $\gamma_1>0$, there exists $c_0\le c_0(\gamma_1)$ small enough such that for all $t \le c_0|\log A_{\e}|$,
\[
\mu_t(A_{\e}^{\beta(1-\sigma)})\le \mu_t(2^k A_{\e}^{\frac{1-\sigma}{2}}) \le A_{\e}^{\gamma_1}.
\]
 Indeed, since
\[\begin{aligned}
\mu_t(2^k A_{\e}^{\frac{1-\sigma}{2}}) \le \frac{|c_0 \log A_{\e}|^k}{k^{k+1/2}}
\le \left|\frac{c_0 \log A_{\e}}{C\log A_{\e}}\right|^k
\le c_0^{C|\log A_{\e}|},
\end{aligned}\]
it suffice to show that
\[
C| \log A_{\e}|\log c_0 \le \gamma_1 \log A_{\e},
\]
which is equivalent to
\[
C \log c_0 \le -\gamma_1
\]
since $\log A_{\e}<0.$
So the claim is true for $c_0\le e^{-\frac{\gamma_1}{C}} $. Finally, for arbitrary $\alpha \in \text{supp}(\om_{0,\e})$, define
\[
R(t):= \max \{ |X(\alpha,t)-p_{\e}(t)| , 8A_{\e}^{\beta(1-\sigma)} \}.
\]
As in section $3$, we obtain
\[
\frac{d}{dt}R \le CR(t),
\]
which implies
\[
R(t) \le 8A_{\e}^{\beta(1-\sigma)} A_{\e}^{-Cc_0}
\]
for all $t \le c_0 |\log A_{\e}|$. Recall that $a < \beta(1-\sigma)$, so we can choose $c_0=c_0(\beta,\sigma,a)$ small enough such that $R(t) \le A_{\e}^a$, which implies
$$\text{supp}(\om_{\e}(\cdot,t))\subset B_{A_{\e}^{a}}(p_{\e}(t))$$
and the proof is complete.

\end{proof}

\subsection{Proof of Theorem \ref{thm long time}}
\begin{proof}

For any $a<\frac12$ we choose $b,\sigma$ such that $a<b<\frac12$ and $a<(1-\sigma)b$. Then as in section $3$, let
 \[\begin{aligned}T_{\e}:=\sup \Big\{&t \in [0,+\infty):\text{supp}(\om_{m,\e}(\cdot,s))\subset B_{\frac{\delta}{16}}(p_{m}(s)), \\&|P_{m,\e}(s)-p_m(s)|+|p_{m,\e}(s)-p_m(s)| \le \frac{\delta}{16} \quad \text{for $m=1, ..., M$ and $0 \le s \le t$} \Big\}. \end{aligned}\]
We assume that $T_{\e} <c_0(b) |\log A_{\e}|$, where $c_0(b)$ is the constant we defined in Theorem \ref{thm section4 main}. Then for any $t \le T_{\e}$, we have
$$|p_{m,\e}(t)-P_{m,\e}(t)|\le CA_{\e}^{1-\sigma}.$$
Thus,
\[
\text{supp}(\om_{m,\e}(\cdot,t))\subset B_{A_{\e}^b+CA_{\e}^{1-\sigma}}(P_{m,\e}(t))\subset B_{2A_{\e}^b}(P_{m,\e}(t)).
\]
Let
\[
G(t)=\sum_{m=1}^M |P_{m,\e}(t)-p_m(t)|,
\]
then compute as in section $3$ gives
$$\frac{d}{dt}G \le C(G(t)+A_{\e}^b),$$
which implies
\[
G(t) \le e^{Ct}(G(0)+tA_{\e}^b)
\]
by Gr\"{o}nwall's inequality. So for all $t<c_0 |\log A_{\e}|$ ,we can take $c_0=c_0(b,C)$ small enough and $\e \le \e_0(c_0)$ small such that
\[
G(t) \le A_{\e}^{b(1-\sigma)}.
\]
As a result,
$$
|P_{m,\e}(t)-p_m(t)|\le A_{\e}^{b(1-\sigma)}.
$$
Together with the assumption $a<b(1-\sigma)$ and the fact
\[
\text{supp}(\om_{m,\e}(\cdot,t))\subset B_{2A_{\e}^b}(P_{m,\e}(t)),
\]
we finally get
\[
\text{supp}(\om_{m,\e}(\cdot,t))\subset B_{A_{\e}^a}(p_{m}(t))
\]
when $\e \le \e_0$ small enough. Then arguing as in section 3, the proof of Theorem \ref{thm long time} is complete.

\end{proof}
\section{Generalization to bounded domains.}  Let $\Omega \subset \R^2$ be an open, bounded and simply connected region with smooth boundary. We will consider in this section the Euler equation in $\Omega$.
\begin{equation}\begin{cases}\begin{aligned}\label{eq euler bounded domain}
\partial_t\om_{\e}+\left(\nabla^{\perp} (\Delta)^{-1}\om_{\e}\right)\cdot \om_{\e} &=0 \,\,\,\,\,\,\,\quad \text{in} \quad\Omega\times [0,+\infty) \\
\om_{\e}(\cdot,0)&=\om_{0,\e} \quad \text{in} \quad\Omega,
\end{aligned}\end{cases}\end{equation}
where
\begin{equation*}
(-\Delta)^{-1}f(x):=\int_{\Omega} G(x,y)f(y)\,dy
\end{equation*}
and
\begin{equation*}
G(x,y)=\frac{1}{2\pi}\log \frac{1}{|x-y|} +H(x,y)
\end{equation*}
is the Green's function on $\Omega$. We are interested in the initial vorticity of the form
$$\om_{0,\e}(x) \approx \sum_{m=1}^M \g_m \delta_{p_m^0},$$
where $\g_1, ..., \g_M$ are $M$ nonzero real numbers and $p^0_1, ... , p^0_M \in \Omega$ are $M$ distinct points. A formal computation suggests that when
\begin{equation*}\om_0(x)=\sum_{m=1}^M \g_m \delta_{p_m^0},\end{equation*}
the solution $\om(x,t)$ of \eqref{eq euler bounded domain} should be
$$
\om(x,t)=\sum_{m=1}^M \g_m \delta_{p_m(t)},
$$
where $p(t)=\big(p_1(t), ..., p_M(t)\big):[0,T^*) \to (\Omega)^M$ solves the equation
\begin{equation*}\begin{cases}\begin{aligned}
\frac{dp_m}{dt}&=\sum_{l \neq m} \frac{\g_l}{2\pi} \frac{(p_m-p_l)^{\perp}}{|p_m-p_l|^2}-\sum_{l=1}^M \gamma_l \nabla_x^{\perp}H(p_m,p_l) \\
p_m(0)&=p_m^0.
\end{aligned}\end{cases}\end{equation*}
See \cite{MaJ}, \cite{MaP2} and \cite{Hel} for references. We aim to show that Theorem \ref{thm main} remains valid in bounded domains.
Let
\[
\om_{0,\e}(x)=\sum_{m=1}^M \om_{0m,\e}(x)+\om_{0p,\e}(x),
\]
then we have

\begin{theorem}\label{thm main bounded domain}
Assume $\om_{0,\e} \in L^{1} \cap L^{\infty}$ and $\om_{0m,\e}$ does not change sign with
\[\text{supp}(\om_{0m,\e}) \subset B_{\e}(p_m(0)).\]
 Let the $m^{th}$ total vorticity $\gamma_{m,\e}$ be $$\gamma_{m,\e}= \int_{\Omega} \om_{0m, \e} \,dx,$$
which satisfies
 \[
 |\gamma_m-\gamma_{m,\e}| \le CA_{\e}^{\frac12}.
 \]
We assume also that $A_{\e} \to 0$ as $\e \to 0$ and there exist some $\gamma >0$ , $p_1>2$ such that $$\| \om_{0m,\e}\|_{p_1} \le A_{\e}^{-\gamma}$$
for $m=1, ... , M$.
Then for any $T < T^*$ and any $a < \frac12$, there exists $\e_0 >0 $ such that for all $0 < \e \le \e_0$, the solution $\om_{\e}$ of \eqref{eq euler bounded domain} admits a decomposition $\om_{\e}=\sum_{m=1}^M \om_{m,\e}+\om_{p,\e}$, where
$$\text{supp}(\om_{m,\e}(\cdot,t))\subset B_{(A_{\e})^{a}}(p_m(t)),$$ $$ \int_{\Omega} \om_{m,\e}(x,t) \,dx=\gamma_{m,\e}$$
and
$$\|\om_{p,\e}\|_r=\|\om_{0p,\e}\|_r$$ for any $r \ge 1$.
\end{theorem}

\begin{proof} The proof is extremely similar to that of Theorem \ref{thm main}, so we only give the ideas here. First, the solution $\om_{\e}(x,t)$ admits a decomposition
\[
\om_{\e}(x,t)=\sum_{m=1}^M \om_{m,\e}(x,t)+\om_{p,\e}(x,t).
\]
As in Section 3, $\om_{p,\e}$ is a small perturbation term and $\om_{m,\e}$ satisfies
\[\begin{cases}\begin{aligned}
\partial_{t}\om_{m,\e}+(u_{m,\e}+F_{1,m,\e}+F_{2,m,\e})\cdot \om_{m,\e}&=0 \\
\om_{m,\e}(\cdot,0)&=\om_{0m,\e},
\end{aligned}\end{cases}\]
where
\[\begin{aligned}
u_{m,\e}(x,t)&:=\frac{1}{2\pi}\int_{\Omega} \frac{(x-y)^{\perp}}{|x-y|^2}\om_{\e,m}(y,t)\,dy, \\
F_{2,m,\e}(x,t)&:=\nabla^{\perp} (\Delta)^{-1}\om_{p,\e}(x,t)
\end{aligned}\]
and
\[
F_{1,m,\e}(x,t):=\sum_{l\neq m}\frac{1}{2\pi}\int_{\Omega}\frac{(x-y)^{\perp}}{|x-y|^2}\om_{l,\e}(y,t)\,dy-\sum_{l=1}^M\int_{\Omega}\nabla^{\perp}H(x,y)\om_{l,\e}(y,t)\,dy.
\]
From \cite{Yud} we see that $|\nabla_xG(x,y)|\lesssim \frac{1}{|x-y|}$, so $F_{2,m,\e}$ remains small as a consequence of Lemma \ref{lemma control of the velocity}. We also note that $F_{1,m,\e}$ is Lipschitz continuous in the support of $\om_{m,\e}$ if $\om_{m,\e}(\cdot,t)$ is supported near $p_m(t)$. So arguing as in Section 2-3 we get the desired conclusion.
\end{proof}
By the same argument as in Section 2-4, Theorem \ref{thm concentration of mass} and Theorem \ref{thm long time} remain valid in bounded domains.

\end{document}